\documentclass[11pt]{article}
\usepackage{amsmath}
\usepackage{amsfonts}
\usepackage{graphicx}
\usepackage{setspace}
\usepackage{amsmath}
\usepackage{amssymb}
\usepackage{latexsym}
\usepackage{amsmath, amsfonts,amssymb, amsthm, euscript,makeidx,color,mathrsfs}

\usepackage[numbers,sort&compress]{natbib}

\def\sqr#1#2{{\vcenter{\vbox{\hrule height.#2pt
              \hbox{\vrule width.#2pt height#1pt \kern#1pt \vrule width.#2pt}
              \hrule height.#2pt}}}}
\usepackage{hyperref}
\hypersetup{
    colorlinks=true,
    linkcolor=blue,
    urlcolor=cyan,
    citecolor=cyan
}
\usepackage{caption}

\RequirePackage[capitalize,nameinlink]{cleveref}

\crefname{section}{section}{sections}
\crefname{subsection}{subsection}{subsections}
\Crefname{section}{Section}{Sections}
\Crefname{subsection}{Subsection}{Subsections}

\crefname{condition}{Condition}{Conditions}

\Crefname{figure}{Figure}{Figures}

\crefformat{equation}{\textup{#2(#1)#3}}
\crefrangeformat{equation}{\textup{#3(#1)#4--#5(#2)#6}}
\crefmultiformat{equation}{\textup{#2(#1)#3}}{ and \textup{#2(#1)#3}}
{, \textup{#2(#1)#3}}{ and \textup{#2(#1)#3}}
\crefrangemultiformat{equation}{\textup{#3(#1)#4--#5(#2)#6}}%
{ and \textup{#3(#1)#4--#5(#2)#6}}{, \textup{#3(#1)#4--#5(#2)#6}}{ and \textup{#3(#1)#4--#5(#2)#6}}

\Crefformat{equation}{#2Equation~\textup{(#1)}#3}
\Crefrangeformat{equation}{Equations~\textup{#3(#1)#4--#5(#2)#6}}
\Crefmultiformat{equation}{Equations~\textup{#2(#1)#3}}{ and \textup{#2(#1)#3}}
{, \textup{#2(#1)#3}}{ and \textup{#2(#1)#3}}
\Crefrangemultiformat{equation}{Equations~\textup{#3(#1)#4--#5(#2)#6}}%
{ and \textup{#3(#1)#4--#5(#2)#6}}{, \textup{#3(#1)#4--#5(#2)#6}}{ and \textup{#3(#1)#4--#5(#2)#6}}

\crefdefaultlabelformat{#2\textup{#1}#3}

\newtheorem {theorem}{Theorem}[section]
\newtheorem {lemma}[theorem]{{\bf Lemma}}

\newtheorem {proposition}[theorem]{{\bf Proposition}}
\theoremstyle{remark}
\newtheorem {remark}{{\bf Remark}}[section]

\newtheorem {condition}{{\bf Condition}}[section]
\theoremstyle{definition}
\newtheorem {definition}{{\bf Definition}}[section]

\theoremstyle{plain} \numberwithin {equation}{section}

\numberwithin{assumption}{section}

\allowdisplaybreaks

\usepackage{enumerate}

\def\ns{\noalign{\medskip}}
\def\ds{\displaystyle}
\def\dbP{{\mathbb{P}}}
\def\dbR{{\mathbb{R}}}

\def\Om{\Omega}
\def\deq{\mathop{\buildrel\Delta\over=}}

\begin{document}

\title{\bf Exact Controllability for a Refined Stochastic Plate Equation\footnote{This work is supported by the NSF of China under grants 12025105, 11971334 and 11931011, by the Chang Jiang Scholars Program from the Chinese Education Ministry, and by the Science Development Project of Sichuan University under grants 2020SCUNL101 and 2020SCUNL201.}}

\author{
    Qi L\"{u}
    \footnote{School of Mathematics, Sichuan University, Chengdu, P. R. China. Email: lu@scu.edu.cn. }
    ~~~ and ~~~
    Yu Wang
    \footnote{School of Mathematics, Sichuan University, Chengdu, P. R. China.
    Email: yuwangmath@163.com.}
}

\date{}

\maketitle

\begin{abstract}
    A widely used stochastic plate equation is the classical plate equation perturbed by a term
    of It\^o's integral. However, it is known that this equation is not exactly controllable even if the controls are
    effective everywhere in both the drift and the diffusion terms and also on the boundary. In
    some sense, this means that some key feature has been ignored in this model. Then, a one-dimensional refined stochastic plate equation is proposed and its exact controllability is established in \cite{Yu2022}. In this paper, by means of a new
    global Carleman estimate, we establish the exact controllability of the multidimensional refined stochastic plate equation with two interior controls and two boundary controls.  Moreover, we give a result about the lack of exact controllability, which
    shows that the action of two interior controls and at least one boundary control is necessary.
\end{abstract}



\pagestyle{plain}

\section{Introduction}

Let $T>0$ and $(\Omega, \mathcal{F}, \mathbf{F},   \mathbb{P})$ with $\mathbf   F=\{\mathcal{F}_{t}\}_{t \geq 0}$ be a complete filtered probability space on which a   one-dimensional standard Brownian motion   $\{W(t)\}_{t \geq 0}$ is defined and   $\mathbf{F}$ is the natural filtration generated   by $W(\cdot)$, augmented by all the $\mathbb{P}$   null sets in $\mathcal{F}$.
Write $ \mathbb{F} $   for the progressive $\sigma$-field with respect   to $\mathbf{F}$.
Let $H$ be a Banach space.
Denote by $L^2_{\mathcal{F}_t}(\Omega;H)$ the   space of all $\mathcal{F}_t$-measurable random   variables $\xi$ such that   $\mathbb{E}|\xi|_H^2<\infty$; by   $L_{\mathbb{F}}^{2}(0, T ; H)$ the space   consisting of all $H$-valued $\mathbf F$-adapted   processes $X(\cdot)$ such that   $\mathbb{E}\bigl(|X(\cdot)|_{L^{2}(0, T ;   H)}^{2}\bigr)<\infty$; by   $L_{\mathbb{F}}^{\infty}(0, T ; H)$  the space   consisting of all $H$-valued $\mathbf F$-adapted   bounded processes;
and by   $C_{\mathbb{F}}([0,T]; L^{2}(\Omega; H))$ the   space consisting of all $H$-valued $ \mathbf{F}   $-adapted processes $X(\cdot)$ such   that $ X(\cdot): [0,T] \rightarrow L^{2}_{\mathcal{F}_{\cdot}}(\Omega;H) $ is continuous.
All these spaces are  Banach spaces with the canonical norms (e.g.,   \cite[Section 2.6]{Lue2021a}).

Let $G \subset \mathbb{R}^{n}$ ($ n \in \mathbb{N} $) be   a bounded domain with a $C^4$ boundary   $\Gamma$. Set $Q=(0, T) \times G$ and $   \Sigma= (0, T) \times \Gamma$. Denote by $\nu(x) =
(\nu^1(x), \cdots, \nu^n(x))$ the unit outward
normal vector of $\Gamma$ at point $x$.

Consider the following refined stochastic plate equation:
\begin{equation}\label{eq.eqIIICon}
    \left\{
        \begin{alignedat}{2}
            &d y = \hat{y} d t +(a_{3} y + f) d W(t) && \quad \text { in } Q, \\
            &d \hat{y} +\Delta^2 y d t = (a_{1} y + a_{2} \cdot \nabla y + a_{5} g)  d t+ (a_{4} y + g)d W(t) && \quad \text { in } Q, \\
            &y= h_{1},  \frac{\partial y}{ \partial \nu}=h_{2} && \quad \text { on } \Sigma, \\
            & (y(0), \hat{y}(0)) = (y_{0},\hat{y}_{0}) && \quad \text{ in } G.
        \end{alignedat}
    \right.
\end{equation}
Here, $ (y_{0}, \hat{y}_{0}) \in H^{-1}(G) \times (H^{3}(G) \cap H^{2}_{0}(G))^{*}$ (where $(H^{3}(G) \cap H^{2}_{0}(G))^{*}$ is the dual space of $H^{3}(G) \cap H^{2}_{0}(G)$ with respect to the pivot space $ L^{2}(G)$),  the coefficients
\begin{align*}
a_{1}, a_{3}, a_{4} \in L^{\infty}_{\mathbb{F}} (0,T; W^{1,\infty}(G)),\quad
a_{2} \in L^{\infty}_{\mathbb{F}} (0,T; W^{2,\infty}(G;\mathbb{R}^{n})),\quad
a_{5} \in L^{\infty}_{\mathbb{F}} (0,T; W^{3,\infty}(G)).
\end{align*}
and the controls
\begin{align*}
(f, g, h_{1}, h_{2}) & \in
L^{2}_{\mathbb{F}} (0,T; H^{-1}(G)) \times
L^{2}_{\mathbb{F}} (0,T; (H^{3}(G) \cap H^{2}_{0}(G))^{*}) \times
L^{2}_{\mathbb{F}} (0,T; L^{2}(\Gamma))
\\
& \quad \times
L^{2}_{\mathbb{F}} (0,T; H^{-1}(\Gamma)).
\end{align*}
\begin{remark}
The term $ a_{5} g $ reflects the influence of the control $ g $ in the diffusion term on the drift term, i.e., if one
puts a control $ g $ in the diffusion term and $ a_{5} g $ will appear as a side effect.
This leads to some technical difficulties in the study of the exact controllability of \cref{eq.eqIIICon}.
\end{remark}

The control system \cref{eq.eqIIICon} is a nonhomegeneous boundary value problem. Its solution is understood in the sense of transposition. For the readers' convenience, we recall it briefly below. A systematic introduction to that can be found in \cite[Section 7.2]{Lue2021a}.

First, we introduce the following reference equation:
\begin{align}
    \label{eq.reference}
    \left\{
        \begin{alignedat}{2}
            &d z = \hat{z} d t + (Z - a_{5} z) d W(t) && \quad \text { in } Q_{\tau}, \\
            &d \hat{z} +\Delta^2 z d t = [ (a_{1} - \operatorname{div} a_{2} - a_{4} a_{5}  ) z - a_{2} \cdot \nabla z  - a_{3} \hat{Z} + a_{4} Z]dt + \hat{Z} d W(t)&& \quad \text { in } Q_{\tau}, \\
            &z= \frac{\partial z}{ \partial \nu}= 0 && \quad \text { on } \Sigma_{\tau}, \\
            & (z(\tau), \hat{z}(\tau)) = (z^{\tau},\hat{z}^{\tau}) && \quad \text{ in } G
            ,
        \end{alignedat}
        \right.
\end{align}
where $ \tau \in (0,T] $, $ Q_{\tau} \deq (0,\tau) \times G $, $ \Sigma_{\tau} \deq (0,\tau) \times \Gamma$, and $ (z^{\tau},\hat{z}^{\tau}) \in L_{\mathcal{F}_{\tau}}^{2}(\Omega; H^{3}(G) \cap H^{2}_{0}(G))  \times L_{\mathcal{F}_{\tau}}^{2}(\Omega; H^{1}_{0} (G)  )$.
By the classical well-posedness result for backward stochastic evolution equations (e.g.,\cite[Section 4.2]{Lue2021a}), we know that  \eqref{eq.reference} admits a unique weak solution
\begin{align*}
    (z,Z, \hat {z}, \hat{Z}) & \in
    L^{2}_{\mathbb{F}}(\Omega;C([0,\tau];( H^{3}(G) \cap H^{2}_{0}(G) )))
    \times L^{2}_{\mathbb{F}}(0,\tau; ( H^{3}(G) \cap H^{2}_{0}(G) ) )
    \\
    & \quad
    \times L^{2}_{\mathbb{F}}(\Omega;C([0,\tau]; H_{0}^{1} (G) ))
    \times L^{2}_{\mathbb{F}}(0,\tau; H_{0}^{1} (G) ).
\end{align*}
Furthermore, for $ 0 \leq s,t \leq \tau $, it holds that
    \begin{equation}\label{prop.energyEst}
    \begin{aligned}
    &|(z(t),\hat{z} (t))|_{L^{2}_{\mathcal{F}_{t}}(\Omega; H^{3}(G) \cap H_{0}^{2}(G) )\times L^{2}_{\mathcal{F}_{t}}(\Omega; H_{0}^{1}(G))}
    \\
    & \leq
    C (
    |(z(s), \hat{z}(s))|_{L^{2}_{\mathcal{F}_{s}}(\Omega; H^{3}(G) \cap H_{0}^{2}(G) )\times L^{2}_{\mathcal{F}_{s}}(\Omega; H_{0}^{1}(G))}\\
    &\qquad
    + |(Z, \hat{Z})|_{L^{2}_{\mathbb{F}}(0, \tau; H^{3}(G) \cap H_{0}^{2}(G))\times  L^{2}_{\mathbb{F}}(0, \tau; H_{0}^{1}(G))}
    )
    .
    \end{aligned}
\end{equation}
Here and in what follows,
we denote by $C$  a generic positive constant depending on $G$, $T$, $\tau$ and $ a_{i} $, $ i=1,\cdots,5 $, whose value may vary from line to line.

Next, we give the following hidden regularity for solutions to \cref{eq.reference}.

\begin{proposition}
    \label{prop.HiddenRegular}
    Let $ (z^{\tau},\hat{z}^{\tau}) \in L_{\mathcal{F}_{\tau}}^{2}(\Omega; H^{3}(G) \cap H^{2}_{0}(G)) \times L_{\mathcal{F}_{\tau}}^{2}(\Omega; H^{1}_{0} (G)  ) $. Then the solution $ (z, Z, \hat{z}, \hat{Z}) $ of \cref{eq.reference} satisfies $ |\nabla \Delta z| |_{\Gamma} \in L^{2} _{\mathbb{F}}(0, \tau;L^{2}(\Gamma))  $. Furthermore,
    \begin{align*}
        |\nabla \Delta z|_{L^{2} _{\mathbb{F}}(0, \tau;L^{2}(\Gamma))}
        & \leq
        C |(z^{\tau},\hat{z}^{\tau})|_{L_{\mathcal{F}_{\tau}}^{2}(\Omega; H^{3}(G) \cap H^{2}_{0}(G))  \times L_{\mathcal{F}_{\tau}}^{2}(\Omega; H^{1}_{0} (G)  ) }
        .
    \end{align*}
\end{proposition}
Proof of Proposition \ref{prop.HiddenRegular} is put in \Cref{sc.2}.

Now we are in a position to give the definition of the transposition solution  to \cref{eq.eqIIICon}.
\begin{definition}
    A pair of stochastic processes $(y, \hat{y}) \in C_{\mathbb{F}} ([0,T];L^{2}(\Omega;H^{-1}(G))) \times$  $
    C_{\mathbb{F}} ([0,T];$  $L^{2}(\Omega;(H^{3}(G) \cap H^{2}_{0}(G))^{*} )) $ is a transposition solution to \cref{eq.eqIIICon} if for any $\tau \in(0, T]$ and $(z^\tau, \hat{z}^\tau) \in L_{\mathcal{F}_\tau}^2(\Omega ; H^{3}(G) \cap H^{2}_{0}(G)) \times$  $ L_{\mathcal{F}_\tau}^2(\Omega ; H^{1}_{0} (G) )$, we have
    $$
    \begin{array}{ll}
     \ds\mathbb{E} \langle\hat{y}(\tau), z^\tau \rangle_{(H^{3}(G) \cap H^{2}_{0}(G))^{*}, H^{3}(G) \cap H^{2}_{0}(G)}- \langle\hat{y}_0, z(0) \rangle_{(H^{3}(G) \cap H^{2}_{0}(G))^{*}, H^{3}(G) \cap H^{2}_{0}(G)}
     \\
     \ns\ds
     -\mathbb{E} \langle y(\tau), \hat{z}^\tau \rangle_{H^{-1}(G), H_{0}^{1}(G)}
            + \langle y_0, \hat{z}(0) \rangle_{H^{-1}(G), H_{0}^{1}(G)} \\
      \ns\ds =-\mathbb{E} \int_0^\tau\langle f, \widehat{Z}\rangle_{H^{-1}(G), H_{0}^{1}(G)} d t+\mathbb{E} \int_0^\tau\langle g, Z\rangle_{(H^{3}(G) \cap H^{2}_{0}(G))^{*}, H^{3}(G) \cap H^{2}_{0}(G)} d t
            \\
      \ns\ds \quad + \mathbb{E} \int_0^\tau \int_{\Gamma} \frac{\partial  \Delta z}{\partial \nu} h_{1} d\Gamma d t
      - \mathbb{E} \int_{0}^{\tau} \langle h_{2} ,  \Delta z \rangle_{H^{-1}(\Gamma), H^{1}(\Gamma)} dt.
     \end{array}
    $$
    Here, $ (z, Z, \hat{z}, \hat{Z}) $ solves \cref{eq.reference}.
\end{definition}

Combining \cref{prop.HiddenRegular} and the well-posedness for  stochastic evolution equation with unbounded control operator in the sense of transposition solution (e.g., \cite[Theorem 7.12]{Lue2021a}), we immediately get the following well-posedness result for \cref{eq.eqIIICon}.
\begin{proposition}
    For each $ (y_{0}, \hat{y}_{0}) \in H^{-1}(G) \times (H^{3}(G) \cap H^{2}_{0}(G))^{*}  $, the system \cref{eq.eqIIICon} admits a unique transposition solution $ (y,\hat{y}) $.
    Moreover,
    \begin{align*}
        &|(y,\hat{y})|_{C_{\mathbb{F}} ([0,T];L^{2}(\Omega;H^{-1}(G))) \times
        C_{\mathbb{F}} ([0,T];L^{2}(\Omega;(H^{3}(G) \cap H^{2}_{0}(G))^{*} ))}
        \\
        & \leq
        C\big(
            |y_{0}|_{H^{-1}(G)}
            + | \hat{y}_{0} |_{(H^{3}(G) \cap H^{2}_{0}(G))^{*}}
             + | f |_{ L^{2}_{\mathbb{F}} (0,T; H^{-1}(G))}
             + | g | _{L^{2}_{\mathbb{F}} (0,T; (H^{3}(G) \cap H^{2}_{0}(G))^{*})}
        \\
        & \quad \quad + | h_{1} |_{ L^{2}_{\mathbb{F}} (0,T; L^{2}(\Gamma)) }
            + | h_{2} |_{L^{2}_{\mathbb{F}} (0,T; H^{-1}(\Gamma))}
        \big).
    \end{align*}
\end{proposition}

Now we give the definition of the exact controllability for \cref{eq.eqIIICon}.

\begin{definition}
    \label{def.control}
    The system \cref{eq.eqIIICon} is called exactly controllable at time $ T $ if for any
    $ (y_{0}, \hat{y}_{0}) \in H^{-1}(G) \times (H^{3}(G) \cap H^{2}_{0}(G))^{*} $
    and
    $ (y_{1}, \hat{y}_{1}) \in L^{2}_{\mathcal{F}_{T}} (\Omega;H^{-1}(G) ) \times  L^{2}_{\mathcal{F}_{T}} (\Omega;(H^{3}(G) \cap H^{2}_{0}(G))^{*} )$, there exist controls
    \begin{align*}
        (f, g, h_{1}, h_{2}) & \in
        L^{2}_{\mathbb{F}} (0,T; H^{-1}(G)) \times
        L^{2}_{\mathbb{F}} (0,T; (H^{3}(G) \cap H^{2}_{0}(G))^{*}) \times
        L^{2}_{\mathbb{F}} (0,T; L^{2}(\Gamma))
        \\
        & \quad \times
        L^{2}_{\mathbb{F}} (0,T; H^{-1}(\Gamma))
    \end{align*}
    such that the solution $ (y, \hat{y}) $ to \cref{eq.eqIIICon} satisfies that $ (y(T, \cdot ), \hat{y}(T, \cdot )) = (y_{1}, \hat{y}_{1}) $, $\dbP$-a.s.
\end{definition}

\begin{remark}
    In the definition of the exact controllability for \cref{eq.eqIIICon}, we put the state space to be   $L^{2}_{\mathcal{F}_{T}} (\Om; H^{-1}(G)) \times  L^{2}_{\mathcal{F}_{T}} (\Om; (H^{3}(G) \cap H^{2}_{0}(G))^{*} )$. It is natural to choose the state space as $L^{2}_{\mathcal{F}_{T}} ( \Om;L^2(G)) \times  L^{2}_{\mathcal{F}_{T}} (\Om;  H^{-2}(G))$. Further, the controls $g\in L^{2}_{\mathbb{F}} (0,T; (H^{3}(G) \cap H^{2}_{0}(G))^{*})$  and $h_{2} \in L^{2}_{\mathbb{F}} (0,T; H^{-1}(\Gamma))$ are very irregular. We expect to use more regular controls to achieve the desired goal. However, we do not know how to do that now. Indeed,  even for the deterministic plate equation, to the best of our knowledge, the existing results (e.g., \cite{Lasiecka1989}) can only prove the exact controllability in the space  $H^{-1}(G) \times (H^{3}(G) \cap H^{2}_{0}(G))^{*}$ with controls in the Dirichlet and Neumann boundary conditions.

\end{remark}

The main result of this paper is the following.

\begin{theorem}
    \label{thm.Control-2}
    The system \cref{eq.eqIIICon} is exactly controllable at any time $T>0$.
\end{theorem}

\begin{remark}

Similar to the derivation process in \cite{Yu2022,Lue2019}, the refined stochastic plate equation \cref{eq.eqIIICon} can be obtained from the classical stochastic plate equation:
\vspace{-1mm}
\begin{equation}\label{eq.conPlate}
    \left\{
        \begin{alignedat}{2}
            &d y_{t}+\Delta^2 y d t= (a_{1} y + a_{2} \cdot \nabla y + f) d t+ (a_{4} y + g) d W(t) && \quad \text { in } Q, \\
            &y = h_{1},  \frac{\partial y}{ \partial \nu}=h_{2} && \quad \text { on } \Sigma, \\
            & (y(0),y_{t}(0)) = (y_{0},y_{1}) && \quad \text{ in } G.
        \end{alignedat}
        \right.
\end{equation}
Here,  $ (y_{0},y_{1}) $ are the initial data, and $ f,g, h_{1}, h_{2} $ are controls.

The system \cref{eq.conPlate} is widely used in structural engineering, and can be applied to beams, bridges and other structures, see \cite{Chow1999,Kim2001,DaPrato2014,Brzezniak2005}.
In particular, \cref{eq.conPlate} can be used to characterize fluttering or large-amplitude vibration of an elastic panel excited by aerodynamic forces which are perturbed by random fluctuations (e.g., \cite{Chow1999}).
However, similar to Theorem 4.1 in \cite{Yu2022} and Theorem 2.1 in \cite{Lue2019}, one can show that  the system \cref{eq.conPlate} is not exactly controllable for any $ T>0 $.
Inspired by the negative controllability result, and similar to \cite{Yu2022,Lue2019}, we study a refined stochastic plate equation \cref{eq.eqIIICon}.
\end{remark}

We put four controls in the system \cref{eq.eqIIICon}. Similarly to Theorem 2.3 in \cite{Lue2019}, one can find that boundary controls $ h_{1} $ and $ h_{2} $ in \cref{eq.eqIIICon} can not be dropped simultaneously, and internal controls $ f $ and $ g $ must be acted on the whole domain $ G $. More precisely, we have the following result.
\begin{theorem}
    \label{thm.negativeControl}
    The system \cref{eq.eqIIICon} is not exactly controllable at any time $ T>0 $ provided that one of the following three conditions is satisfied:
    \begin{enumerate}[(1).]
        \item $ a_{3} \in C_{\mathbb{F}} ([0,T]; L^{\infty}(G)) $, $ G \backslash \overline{G_{0}} \neq \emptyset $ and $ \operatorname{supp} f \subset G_{0} $;
        \item $ a_{4} \in C_{\mathbb{F}} ([0,T]; L^{\infty}(G)) $, $ G \backslash \overline{G_{0}} \neq \emptyset $ and $ \operatorname{supp} g \subset G_{0} $;
        \item $ h_{1} = h_{2} = 0 $.
    \end{enumerate}
\end{theorem}

\begin{remark}
    It is worth studying whether one the of boundary controls can be removed.
    This can be done for deterministic plate equation (e.g., \cite{Lasiecka1989,Lions1988a}). However, we have no idea for how to do that.
\end{remark}

\begin{remark}
    By letting $ G=(0,1) $, we can deduce from \cref{thm.ceII} that the system \cref{eq.eqIIICon} is exactly controllable with controls in any nonempty subset of the boundary $ \Gamma $, which has recently been proved in \cite{Yu2022}.
    In fact, thanks to \cref{rk.oneObe}, for any $  (y_{0}, \hat{y}_{0}) $ and $  (y_{1}, \hat{y}_{1}) $ satisfing \cref{def.control}, one can find $ (f, g, h_{1},h_{2}) \in L^{2}_{\mathbb{F}} (0,T; H^{-1}(G)) \times
    L^{2}_{\mathbb{F}} (0,T; (H^{3}(G) \cap H^{2}_{0}(G))^{*}) \times (L^{2}_{\mathbb{F}}(0,T))^{2} $ such that the solution $ (y, \hat{y}) $  to \cref{eq.eqIIICon}, where the boundary conditions are
    \begin{align*}
        y(\cdot,0) = h_{1}, \quad y_{x}(\cdot,0) = h_{2}, \quad
        y(\cdot,1) = 0, \quad y_{x}(\cdot,1) = 0, \quad \text{~ on } (0,T)
    \end{align*}
    satisfies that $ (y(T, \cdot ), \hat{y}(T, \cdot )) = (y_{1}, \hat{y}_{1}) $.
    In the multidimensional case, it is  worth to studying whether \cref{eq.eqIIICon} is still exactly controllable under the assumption that  $ (h_1, h_2) \in (L^{2}_{\mathbb{F}}(0,T;L^{2}(\Gamma_{0})))^{2} $, where $ \Gamma_{0} $ is a nonempty subset of $ \Gamma $.
\end{remark}

By a standard duality argument,  \cref{thm.Control-2}  is equivalent to the following observability estimate (e.g., \cite[Theorem 7.17]{Lue2021a}).

\begin{theorem}\label{thm.ObeF}
    There exists a constant $ C>0 $  such that for every $ (z^{T}, \hat{z}^{T}) \in L_{\mathcal{F}_T}^2 (\Omega; H^{3}(G)\cap H^{2}_{0}(G)) \times L_{\mathcal{F}_T}^2(\Omega ; H^{1}_{0} (G) )$, it holds that
    \begin{align*}
        &|(z^{T},\hat{z}^{T})|^{2}_{ L_{\mathcal{F}_T}^2 (\Omega ;H^{3}(G) \cap H^{2}_{0}(G)) \times L_{\mathcal{F}_T}^2(\Omega ; H^{1}_{0} (G) )}
        \\
        &\leq
        C \mathbb{E} \int_{\Sigma} (|\nabla \Delta z|^{2} + |\Delta z|^{2}) d\Gamma dt
        + C |(Z, \hat{Z})|^{2}_{L^{2}_{\mathbb{F}}(0, T;H^{3}(G) \cap H^{2}_{0}(G) )\times L^{2}_{\mathbb{F}}(0, T; H^{1}_{0} (G) )}
        ,
    \end{align*}
    where $  (z,Z, \hat {z}, \hat{Z})  $ is the solution to the equation \cref{eq.reference} with $ \tau = T $, $ z(T) = z^{T} $, and $ \hat{z}(T) = \hat{z}^{T}$.
\end{theorem}
\begin{remark}
A sharp trace estimate for the deterministic plate equation is established in \cite{Lasiecka1993}. It suggests that we may need better regularity than $H_{0}^2(G)\times L^2(G)$ for the initial datum of dual system \cref{eq.reference} to estimate $|\nabla \Delta z|$ on the boundary. This is the reason that we choose the final data of \cref{eq.reference} in $ L_{\mathcal{F}_T}^2 (\Omega ; H^{3}(G)\cap H^{2}_{0}(G)) \times L_{\mathcal{F}_T}^2(\Omega ; H^{1}_{0} (G) )$.
\end{remark}

\begin{remark}\label{rk.oneObe}
    Let $ G=(0,1) $. By choosing $ x_{0} > 1 $ in \cref{eq.xiEta}, from the proof of \cref{thm.ceII,thm.ObeF}, we can deduce that
    \begin{align*}
       &|(z^{T},\hat{z}^{T})|^{2}_{ L_{\mathcal{F}_T}^2 (\Omega ;H^{3}(G) \cap H^{2}_{0}(G)) \times L_{\mathcal{F}_T}^2(\Omega ; H^{1}_{0} (G) )}
        \\
        &\leq
        C \mathbb{E} \int_{0}^{T} \big(|z_{xxx}(0)|^{2} + |z_{xx}(0)|^{2}\big) dt
        + C |(Z, \hat{Z})|^{2}_{L^{2}_{\mathbb{F}}(0, T; H^{3}(G) \cap H^{2}_{0}(G) )\times L^{2}_{\mathbb{F}}(0, T; H^{1}_{0} (G) )}
        .
    \end{align*}
    This is the main result in \cite{Yu2022}.
\end{remark}

\begin{remark}
    One can also consider the exact controllability of a refined stochastic palte equation when the boundary controls are as follows:
    \begin{align*}
        y = h_{1}  \quad \mbox{ and } \quad \Delta y = h_{2} \text{~ on ~} \Sigma
        .
    \end{align*}
    By the standard duality argument and the technique of transforming the controllability of the forward stochastic equation into the controllability of a backward equation (e.g., \cite[Section 7.5]{Lue2021a}), we only need to prove that for any $(z_0,\hat z_0) \in  \{\eta\in H^3(G)| \Delta \eta  \in  H^{1}_{0} (G)\} \times H^{1}_{0} (G)  $, it holds that
    \begin{align}\label{eq.oBE02-1}
    |(\Delta z_{0},\hat{z}_{0})|^{2}_{H_{0}^{1}(G) \times H_{0}^{1}(G)}
    \leq
    C \mathbb{E} \int_{0}^{T} \int_{\Gamma} \bigg (\bigg | \frac{\partial \Delta z}{ \partial \nu } \bigg |^{2}
    + \bigg |\frac{\partial z}{ \partial \nu }\bigg |^{2}
    \bigg ) d\Gamma dt,
    \end{align}
    where $(z,\hat z)$ is the solution to
    \begin{equation}
        \label{eq.dualSCB1}
        \left\{
            \begin{alignedat}{2}
                &d z = \hat{z} d t - a_{5} z d W(t) && \quad \text { in } Q, \\
                &d \hat{z} +\Delta^2 z d t = [ (a_{1} - \operatorname{div} a_{2} ) z - a_{2} \cdot \nabla z - a_{4} a_{5} z]dt && \quad \text { in } Q, \\
                &z= \Delta z= 0 && \quad \text { on } \Sigma, \\
                & (z(0), \hat{z}(0)) = (z_{0},\hat{z}_{0}) && \quad \text{ in } G.
            \end{alignedat}
            \right.
    \end{equation}
    Following the idea in \cite{Zhang2001}, we can rewrite the  equation \eqref{eq.dualSCB1} as two coupled stochastic Schr\"{o}dinger equations, and the desired observability estimate can be obtained from the Carleman estimate for the latter.
    In fact, letting $ u = i \hat{z} + \Delta z $, we have
    \begin{align*}
        i dz + \Delta z dt
        = u dt - i a_{5} z d W(t),
    \end{align*}
    and
    \begin{align*}
        - i d u + \Delta u dt
        & = [ (a_{1} - \operatorname{div} a_{2} ) z - a_{2} \cdot \nabla z - a_{4} a_{5} z] dt + i \Delta (a_{5} z) d W(t).
    \end{align*}
    Clearly, it holds that $ z = u = 0 $ on $ \Sigma $.
    Thanks to Theorem 1.2 in \cite{Lue2013},  we can obtain that
    \begin{align}\label{eq.oBE02}
        |(\Delta z_{0},\hat{z}_{0})|^{2}_{H_{0}^{1}(G) \times H_{0}^{1}(G)}
        \leq
        C \mathbb{E} \int_{0}^{T} \int_{\Gamma} \bigg (\bigg | \frac{\partial \Delta z}{ \partial \nu } \bigg |^{2}
        + \bigg |\frac{\partial z}{ \partial \nu }\bigg |^{2}
        + \bigg |\frac{\partial \hat{z}}{ \partial \nu }\bigg |^{2}
        \bigg ) d\Gamma dt
        .
    \end{align}
    Then one may follow the rechnique in \cite{Lasiecka1990} to derive \eqref{eq.oBE02-1} from \eqref{eq.oBE02}. On the other hand, one can also follow the technique in this paper to prove the desired observability \eqref{eq.oBE02-1}.
\end{remark}

There are a large number of published works (\cite{Ball1982,Hansen2011,Haraux1989,Jaffard1990,Lasiecka1989,Lions1988a,Liu1998,Puel1993,Zhang2001,Eller2015} and the references therein) studying the exact controllability for deterministic plate equations.
However, as far as we know, \cite{Yu2022} is the only published work investigating the exact controllability of stochastic beam equations, in which the authors show that \cref{eq.eqIIICon} is exactly controllable when $G$ is an interval. In this paper, we shall use a stochastic Carleman estimate to prove \cref{thm.ObeF}.  Such kind of estimate is one of the most useful tool in studying controllability for stochastic partial differential equations (see \cite{FL1,Gao2015,Lue2014,Lue2013,Lue2021a,Wang,Tang2009} and the references given there). Nevertheless,  \cite{Yu2022} is the only published work using this method to study  the exact controllability of stochastic beam equations.

The rest of this paper is organized as follows. In  \Cref{sc.2}, we provide some preliminaries. \Cref{sc.3} is devoted to establishing a Carleman estimate for the adjoint equation \cref{eq.reference}. By means of that Carleman estimate,
we prove \cref{thm.ObeF} in  \Cref{sc.4}. At last,  \Cref{sc.5} is addressed to the proof of   \cref{thm.negativeControl}.

\section{Some preliminary results}\label{sc.2}

This section provides some preliminary results. In the rest of this paper, the notation $y_{x_{i}} \equiv y_{x_{i}}(x)=\partial y(x) / \partial x_{i}$ will be used for simplicity, where $x_{i}$ is the $i$-th coordinate of a generic point $x=(x_{1}, \cdots, x_{n})$ in $\mathbb{R}^{n} $. In a similar manner, we
use notations $z_{x_j}$, $v_{x_j}$, etc. for the
partial derivatives of $z$ and $v$ with respect
to $x_j$. 

We first prove the hidden regularity for solutions to \cref{eq.reference}.

\begin{proof}[Proof of \cref{prop.HiddenRegular}]

    For any $ \rho \deq (\rho^{1}, \cdots, \rho^{n}) \in C^{2}(\mathbb{R}^{n+1}; \mathbb{R}^{n}) $, by It\^o's formula and \cref{eq.reference}, we have
    \begin{equation}
    \label{eq.identityF}
    \begin{array}{ll}\ds
      \sum_{j=1}^{n} ( 2\rho \cdot \nabla \Delta z \Delta z_{x_{j}} - \rho^{j} |\nabla \Delta z|^{2} ) _{x_{j}} dt
    \\
    \ns\ds = - \operatorname{div} \rho |\nabla \Delta z |^{2} dt
    + 2 \rho \cdot \nabla \Delta z \big[(a_{1} - \operatorname{div} a_{2} -a_{4} a_{5}) z - a_{2} \cdot \nabla z + a_{4} Z - a_{3} \hat{Z}\big] dt
    \\
    \ns\ds \quad
    + 2\rho \cdot \nabla \Delta z \hat{Z} d W(t)
    + 2  \nabla \Delta z D \rho \cdot \nabla \Delta zdt
    -2 d(\rho \cdot \nabla \Delta z \hat{z})
    + 2 \rho_{t} \cdot \nabla \Delta z \hat{z}dt
    \\
    \ns\ds  \quad
    + 2 \rho \cdot \nabla \Delta (Z - a_{5}z) \hat{Z} dt
    + 2 \rho\cdot \nabla \Delta( Z - a_{5}z) \hat{z} d W(t)
    + 2 \operatorname{div} (\hat{z} \rho \nabla^{2} \hat{z} ) dt
    \\
    \ns\ds \quad
    - 2 \operatorname{div} (\hat{z} \nabla \hat{z} D \rho) dt
    + 2 \Delta \rho \cdot \nabla \hat{z} \hat{z} dt
    + 2 \nabla \hat{z} D \rho\cdot \nabla \hat{z} dt
    - \operatorname{div} (\rho|\nabla \hat{z}|^{2})dt
    + \operatorname{div}\rho|\nabla \hat{z}|^{2}dt
    .
    \end{array}
    \end{equation}
    Since $ \Gamma \in C^{3} $, there exists a vector field $ \zeta \in C^{2}(\mathbb{R}^{n};\mathbb{R}^{n}) $ such that $ \zeta  = \nu $ on $ \Gamma $ (e.g., \cite[Lemma 2.1]{Komornik1994}).
    Setting $ \rho= \zeta  $, integrating \cref{eq.identityF} in $ Q_{\tau} $, and taking expectation on $ \Omega $, we have
    \begin{align*}
        & \mathbb{E} \int_{\Sigma_{\tau}} \biggl(
            2 \bigg| \frac{\partial \Delta z}{\partial \nu} \bigg|^{2} 
            - |\nabla \Delta z|^{2}
        \biggr) d\Gamma dt\\
        & 
        = \mathbb{E} \int_{\Sigma_{\tau}}  \sum_{j=1}^{n} ( 2\rho \cdot \nabla \Delta z \Delta z_{x_{j}} - \rho^{j} |\nabla \Delta z|^{2} ) \nu^{j} d\Gamma dt
        \\
        & = - 2 \mathbb{E} \int_{G} \rho \cdot \nabla \Delta z^{\tau} \hat{z}^{\tau} d x
        +  2 \mathbb{E} \int_{G} \rho \cdot \nabla \Delta z(0) \hat{z}(0) d x
        \\
        & \quad
    + \mathbb{E} \int_{Q_{\tau}} \bigl\{
            - \operatorname{div} \rho |\nabla \Delta z |^{2} 
    + 2 \rho \cdot \nabla \Delta z \big[(a_{1} - \operatorname{div} a_{2} -a_{4} a_{5}) z - a_{2} \cdot \nabla z + a_{4} Z - a_{3} \hat{Z}\big] 
    \\
    & \qquad \qquad \quad
    + 2  \nabla \Delta z D \rho \cdot \nabla \Delta z
    + 2 \rho_{t} \cdot \nabla \Delta z \hat{z}
    + 2 \rho \cdot \nabla \Delta (Z - a_{5}z) \hat{Z} 
    + 2 \Delta \rho \cdot \nabla \hat{z} \hat{z} 
    \\[3mm]
    & \qquad \qquad \quad
    + 2 \nabla \hat{z} D \rho\cdot \nabla \hat{z} 
    + \operatorname{div}\rho|\nabla \hat{z}|^{2}
    \bigr\} dxdt
    .
    \end{align*}
    This implies
    \begin{align} \label{eq.6.3.1F}
    &2 \bigg| \frac{\partial \Delta z}{\partial \nu} \bigg|^{2}_{L^{2} _{\mathbb{F}}(0, \tau;L^{2}(\Gamma))}
    - |\nabla \Delta z|^{2}_{L^{2} _{\mathbb{F}}(0, \tau;L^{2}(\Gamma))}
    \leq  C |(z^{\tau},\hat{z}^{\tau})|^{2}_{L_{\mathcal{F}_{\tau}}^{2}(\Omega; H^{3}(G) \cap H_{0}^{2}(G)) \times L_{\mathcal{F}_{\tau}}^{2}(\Omega; H_{0}^{1}(G) ) }
    .
    \end{align}

    Denote by $ \nabla_{\sigma} $ the tangential gradient on $ \Gamma $. We have
    \begin{align*}
    |\nabla \Delta z|^{2} =  \bigg| \frac{\partial \Delta z}{\partial \nu} \bigg|^{2} + |\nabla_{\sigma} \Delta z |^{2},
    \end{align*}
    which, together with \cref{eq.6.3.1F}, implies that
    \begin{align*}
    \bigg| \frac{\partial \Delta z}{\partial \nu} \bigg|^{2}_{L^{2} _{\mathbb{F}}(0, \tau;L^{2}(\Gamma))}
    \leq
    |\nabla_{\sigma} \Delta z |^{2}_{L^{2} _{\mathbb{F}}(0, \tau;L^{2}(\Gamma))}
    + C |(z^{\tau},\hat{z}^{\tau})|^{2}_{L_{\mathcal{F}_{\tau}}^{2}(\Omega; H^{3}(G) \cap H_{0}^{2}(G)) \times L_{\mathcal{F}_{\tau}}^{2}(\Omega; H_{0}^{1}(G) ) }
    .
    \end{align*}
    Now we are going to  prove
    \begin{align}\label{eq.1123final}
    |\nabla_{\sigma} \Delta z |^{2}_{L^{2} _{\mathbb{F}}(0, \tau;L^{2}(\Gamma))}
    \leq
    C|(z^{\tau},\hat{z}^{\tau})|^{2}_{L_{\mathcal{F}_{\tau}}^{2}(\Omega; H^{3}(G) \cap H_{0}^{2}(G)) \times L_{\mathcal{F}_{\tau}}^{2}(\Omega; H_{0}^{1}(G) ) }
    .
    \end{align}
    As the proof of Theorem 2.2 in \cite{Lasiecka1986}, we introduce the following operator
    \begin{equation}\label{11.23-eq1}
        \left\{ 
        \begin{alignedat}{1}
            \mathscr{B} &= \sum_{i=1}^{n} b_{i}(x) \frac{\partial }{\partial x_{i}} = \text{ a first-order differential operator (time independent)}  \\
            & \text {with coefficients $ b_{i} \in C^{4}(\overline{G}) $ and such  that $ \mathscr{B}  $  is tangential to $ \Gamma$, i.e.,} \\
            & \sum_{i=1}^{n} b_{i}(x) \nu^{i} = 0  \text{ on }  \Gamma.
        \end{alignedat}
        \right.
    \end{equation}
The operator $ \mathscr{B} $ can be thought of as the pre-image, under the diffeomorphism via partitions of unity from $ G $ onto  half-space $ \{ (x,y) \in \mathbb{R}^{n} \mid x >0, y \in \mathbb{R}^{n-1} \} $ of the tangential derivative on the boundary $ x=0 $.

    Define
    \begin{align*}
    p \deq \mathscr{B} z \in C_{\mathbb{F}}([0,\tau]; L^{2}(\Omega; H^{2}(G)) ), \qquad
    \hat{p} \deq \mathscr{B} \hat{z} \in C_{\mathbb{F}}([0,\tau]; L^{2} (\Omega; L^{2}(G))).
    \end{align*}
    From \cref{eq.reference}, we have
    \begin{equation}\label{eq.6.3.2F}
    \left\{
    \begin{alignedat}{2}
    &d p = \hat{p} d t + g_{1} d W(t) && \quad \text { in } Q_{\tau}, \\
    &d \hat{p} +\Delta^2 p d t = f_{2} dt + g_{2} d W(t) && \quad \text { in } Q_{\tau}, \\
    &p=0,  \frac{\partial p}{ \partial \nu}= \bigg[ \frac{\partial}{\partial \nu}, \mathscr{B} \bigg] z && \quad \text { on } \Sigma_{\tau}, \\
    & (p(\tau), \hat{p}(\tau)) = ( \mathscr{B} z^{\tau}, \mathscr{B} \hat{z}^{\tau}) && \quad \text{ in } G,
    \end{alignedat}
    \right.
    \end{equation}
    where
    $$
    g_{1}=  \mathscr{B} (Z- a_{5} z),\quad g_{2} =  \mathscr{B}  \hat{Z},
    $$
    and
    $$
    f_{2} = \mathscr{B} \big[ (a_{1} - \operatorname{div} a_{2} -a_{4} a_{5} ) z - a_{2} \cdot \nabla z - a_{3} \hat{Z} + a_{4} Z\big] + [\Delta^{2}, \mathscr{B}] z,
    $$
    and $[\cdot,\cdot]$ denotes the commutator of two operators.

    From \eqref{11.23-eq1}, we get that 
    \begin{align}\label{eq.1123f}
    | [\Delta^{2}, \mathscr{B}] z |_{L^{2}_{\mathbb{F}} (0,\tau; H^{-1}(G))}
    \leq C|(z^{\tau},\hat{z}^{\tau})|_{L_{\mathcal{F}_{\tau}}^{2}(\Omega; (H^{3}(G) \cap H_{0}^{2}(G) ) ) \times L_{\mathcal{F}_{\tau}}^{2}(\Omega; H_{0}^{1}(G) ) }
    \end{align}
    and that 
    \begin{equation}\label{11.23-eq2}
    \begin{array}{ll}\ds
    \mathbb{E} \int_{\Sigma_{\tau}} |\nabla_{\sigma} \Delta z |^{2} d\Gamma dt
     =
    \mathbb{E} \int_{\Sigma_{\tau}} |\mathscr{B} \Delta z |^{2} d\Gamma dt
    =
    \mathbb{E} \int_{\Sigma_{\tau}} | \Delta p |^{2} d\Gamma dt
    + \mathbb{E} \int_{\Sigma_{\tau}} | [\mathscr{B}, \Delta ]    z |^{2} d\Gamma dt
    .
    \end{array}
    \end{equation}
    By \eqref{11.23-eq1} again, we find that 
    \begin{equation}\label{11.23-eq3}
    \begin{array}{ll}\ds
    \mathbb{E} \int_{\Sigma_{\tau}} | [\mathscr{B}, \Delta ]z |^{2} d\Gamma dt &\ds \leq C\mathbb{E} \int_{0}^\tau |z|_{H^{3}(G)}^{2} dt\\
    \ns&\ds
    \leq C |(z^{\tau},\hat{z}^{\tau})|^{2}_{L_{\mathcal{F}_{\tau}}^{2}(\Omega; H^{3}(G) \cap H_{0}^{2}(G)) \times L_{\mathcal{F}_{\tau}}^{2}(\Omega; H_{0}^{1}(G) ) }.
    \end{array}
\end{equation}
   Combining \cref{11.23-eq2,11.23-eq3}, to show \cref{eq.1123final}, we only need to prove
    \begin{align}
    \label{eq.6.3.2-1F}
    \mathbb{E} \int_{\Sigma_{\tau}} | \Delta p |^{2} d\Gamma dt
    \leq C |(z^{\tau},\hat{z}^{\tau})|^{2}_{L_{\mathcal{F}_{\tau}}^{2}(\Omega; (H^{3}(G) \cap H_{0}^{2}(G) ) ) \times L_{\mathcal{F}_{\tau}}^{2}(\Omega; H_{0}^{1}(G) ) }.
    \end{align}

    For any $ \rho \deq (\rho^{1}, \cdots, \rho^{n}) \in C^{2}(\mathbb{R}^{n+1}; \mathbb{R}^{n}) $, by It\^o's formula and \cref{eq.6.3.2F}, we have
    \begin{equation}\label{eq.2OrderIdentityF}
    \begin{array}{ll}\ds  
     \sum_{i,k,l =1}^{n} (\rho^{i}p_{x_{k}x_{k}} p_{x_{l}x_{l}})_{x_{i}} dt
    \\[-1mm] 
    \ns\ds =
    \sum_{i,k,l =1}^{n} \big[
    \rho^{i}_{x_{i}}p_{x_{k}x_{k}} p_{x_{l}x_{l}}
    +  (2 \rho^{i} p_{x_{i}x_{k}x_{k} } p_{x_{l}})_{x_{l}}
    -  (2 \rho^{i}_{x_{l}} p_{x_{i}x_{k} } p_{x_{l}})_{x_{k}}
    + 2 \rho^{i}_{x_{l}x_{k}} p_{x_{i}x_{k}} p_{x_{l}}
    \\ 
    \ns\ds \qquad \quad ~~
    + 2 \rho^{i}_{x_{l}} p_{x_{i}x_{k}} p_{x_{l} x_{k} }
    - (2 \rho^{i} p_{x_{l}x_{k}x_{k}} p_{x_{l}})_{x_{i}}
    + (2 \rho^{i}_{x_{i}}  p_{x_{l}x_{k}} p_{x_{l}})_{x_{k}}
    -  2 \rho^{i}_{x_{i}x_{k}} p_{x_{l}x_{k} } p_{x_{l}}
    \\[3mm] 
    \ns\ds \qquad \quad ~~
    -  2 \rho^{i}_{x_{i}} p_{x_{l}x_{k} } p_{x_{l}x_{k}}
    + (2 \rho^{i}  p_{x_{l}x_{k} x_{k}} p_{x_{i}})_{x_{l}}
    - (2 \rho^{i}_{x_{l}}  p_{x_{l}x_{k} } p_{x_{i}})_{x_{k}}
    +  2 \rho^{i}_{x_{l}x_{k}} p_{x_{l}x_{k} }p_{x_{i}}
    \\[3mm]  
    \ns\ds \qquad \quad ~~
    +  2 \rho^{i}_{x_{l}} p_{x_{l}x_{k}}  p_{x_{i}x_{k}}
    \big] dt
    - 2 \rho\cdot \nabla p (f_{2} dt + g_{2} dW(t) )
    + d (2 \rho \cdot \nabla p \hat{p})
    - 2 \rho_{t} \cdot \nabla p \hat{p} dt
    \\[3mm]  
    \ns\ds \quad  +
    \operatorname{div}\rho \hat{p}^{2} dt
    - 2 \rho\cdot d\nabla p d \hat{p}
    - \operatorname{div} (\rho\hat{p}^{2}) dt
    - 2 \rho  \cdot\nabla g_{1} \hat{p} d W(t).
    \end{array}
\end{equation}
    Setting $\rho= \zeta $, integrating \cref{eq.2OrderIdentityF} in $ Q_{\tau} $, and taking expectation on $ \Omega $, we have 
    \begin{equation}
        \begin{aligned}
            & \mathbb{E} \int_{\Sigma_{\tau}} |\Delta p|^{2} d \Gamma dt \\
            &=  \mathbb{E} \int_{\Sigma_{\tau}} \sum_{i,k,l =1}^{n} \rho^{i} \nu^{i} p_{x_{k}x_{k}} p_{x_{l}x_{l}}  d\Gamma dt
            \\[1mm]
            &=2 \mathbb{E} \int_{G}  \rho\cdot \nabla p^{\tau} \hat{p}^{\tau} dx - 2\mathbb{E} \int_{G}  \rho\cdot \nabla p(0) \hat{p}(0) dx \\
            & \quad
            + \mathbb{E} \int_{Q_{\tau}} \big( 
             \operatorname{div}\rho \hat{p}^{2}
             -2 \rho\cdot \nabla p f_{2} 
             -2  \rho_{t}\cdot \nabla p \hat{p} 
             -2 \rho\cdot \nabla g_{1} g_{2}\big) dx dt
         \\
         & \quad 
         + \mathbb{E} \int_{Q_{\tau}} \sum_{i,k,l =1}^{n} \big( 
             \rho^{i}_{x_{i}}p_{x_{k}x_{k}} p_{x_{l}x_{l}}
              +  2 \rho^{i}_{x_{l}x_{k}} p_{x_{i}x_{k}} p_{x_{l}}
             + 2 \rho^{i}_{x_{l}} p_{x_{i}x_{k}} p_{x_{l} x_{k} }
            -  2 \rho^{i}_{x_{i}x_{k}} p_{x_{l}x_{k} } p_{x_{l}} 
             \\[2mm]
             & \qquad \qquad \qquad \quad ~~-  2 \rho^{i}_{x_{i}} p_{x_{l}x_{k} } p_{x_{l}x_{k}}
             +  2 \rho^{i}_{x_{l}x_{k}} p_{x_{l}x_{k} }p_{x_{i}}
             +  2 \rho^{i}_{x_{l}} p_{x_{l}x_{k}}  p_{x_{i}x_{k}}
         \big) dx dt
         \\[3mm]
         & \quad 
         + 2 \mathbb{E} \int_{\Sigma_{\tau}} \sum_{i,k,l =1}^{n} \big( 
                \rho^{i}_{x_{i}}  p_{x_{l}x_{k}} p_{x_{l}} \nu^{k}
             -   \rho^{i}_{x_{l}} p_{x_{i}x_{k} } p_{x_{l}} \nu^{k}
             +   \rho^{i}  p_{x_{l}x_{k} x_{k}} p_{x_{i}} \nu^{l}
             -  \rho^{i}_{x_{l}}  p_{x_{l}x_{k} } p_{x_{i}} \nu^{k}
         \big) d\Gamma dt,
         \end{aligned}
    \end{equation}

which, together with \cref{eq.1123f,11.23-eq1,eq.6.3.2F}, implies \cref{eq.6.3.2-1F}. Then we complete the proof.
\end{proof}

Next, we give a pointwise weighted identity,  which will play an important role in the proof of \cref{thm.ceII}.

We have the following fundamental identity.

\begin{theorem}\label{thm.fundamentalIndentity-3}
    Let $ v $ be an $ H^{4}(G) $-valued It\^o process and  $ \hat{v} $ be an $ H^{2}(G) $-valued It\^o process such that
    \begin{equation*}
        dv = (\hat{v}  + f_{1}) dt + g_{1} d W(t)
    \end{equation*}
    for some $ f_{1} \in L^{2}_{\mathbb{F}} (0,T;H^{2}(G)) $ and $ g_{1} \in L^{2}_{\mathbb{F}} (0,T;H^{4}(G) \cap H_{0}^{2}(G)). $
Let $\eta\in C^2(\dbR\times\dbR^n)$.    Set $ \theta=e^{\ell} $, $ \ell=s\xi $, $ \xi=e^{\lambda \eta} $, $ w=\theta v $,  and  $\hat{w} = \theta \hat{v} + \ell_{t} w$. Then, for any $ t\in[0,T] $ and a.e. $ (x,\omega) \in G \times \Omega $,
    \begin{align}
        \notag
        & 2 \theta I_{2} (  d \hat{v} + \Delta^{2} v dt )
        - 2 \operatorname{div} (V_{1}+V_{2}) dt
        \\[3mm] \label{eq.finalEquation-3}
        & =
        2 I_{2}^{2} dt
        + 2 I_{2}I_{3}
        + 2 (M_{1}+M_{2})dt
        + 2 \sum_{i,j,k,l=1}^n \Lambda^{ijkl}_{1} w_{x_{i}x_{j}}w_{x_{k}x_{l}}dt
        + 2 \sum_{i, j=1}^n \Lambda_{2}^{ij} w_{x_{i}} w_{x_{j}}dt
        \\ \notag
        & \quad
        + 2 \Lambda_{3} w^{2}dt
        + 2 \Lambda_{4}
        + 2 d(I_{2} \hat{w})
        - \sum_{i=1}^{n} ( \Phi_{1}^{i} \hat{w} d \Delta w ) _{x_{i}}.
    \end{align}
    Here
    \begin{equation}\label{eq.I-2}
    	I_{1} = \Delta^{2} w dt  + \Psi_{2}\Delta w dt + \sum_{i, j=1}^n  \Psi_{3} ^{ij} w_{x_{i}x_{j}} dt + \sum_{i=1}^n \Psi_{4}^{i}  w_{x_{i}} dt+ \sum_{i=1}^n \Psi_{5}^{i}  w_{x_{i}}  dt + \Psi_{6} w dt + d \hat{w},
    \end{equation}
    with
    \begin{equation}\label{eq.Ic-2}
    \begin{cases}
     \Psi_{2} = 2 s^{2}\lambda^{2}\xi^{2} |\nabla\eta|^{2},
    	\quad
    	\Psi_{3}^{ij} = 4 s^{2}\lambda^{2}\xi^{2} \eta_{x_{i}} \eta_{x_{j}},
    	\\ \ns\ds
    	\Psi_{4}^{i} = 12 s^{2}\lambda^{3}\xi^{2} |\nabla\eta|^{2} \eta_{x_{i}}  +4 s^{2} \lambda^{2} \xi^{2} \Delta \eta \eta_{x_{i}} ,
    	\\  
    	\ns\ds
    	\Psi_{5}^{i} =  \sum_{j=1}^n 8 s^{2}\lambda^{2}\xi^{2} \eta_{x_{i}x_{j}} \eta_{x_{j}},
    	\quad
    	\Psi_{6} = s ^{4}\lambda ^{4}\xi^{4} |\nabla \eta|^{4},
    \end{cases}\quad i,j=1,\cdots,n
\end{equation}
\begin{equation}\label{eq.I2-2}
 	I_{2}  =    \sum_{i=1}^n \Phi_{1}^{i} \Delta w_{x_{i}}  + \Phi_{2} \Delta w  + \sum_{i, j=1}^n \Phi_{3}^{ij}  w_{x_{i}x_{j}}  + \sum_{i=1}^n \Phi_{4} ^{i}  w_{x_{i}}  + \Phi_{5} w
 	,
 \end{equation}
 with
\begin{equation}\label{eq.I2c-2}
	\begin{cases}\ds \Phi_{1}^{i} = -4 s \lambda \xi  \eta_{x_{i}},
 	\quad
 	\Phi_{2} =  -2 s\lambda^{2}\xi |\nabla \eta|^{2} - 2s\lambda\xi  \Delta \eta -  \lambda,
 	\\ \ns\ds
 	\Phi_{3}^{ij} = 4 s\lambda\xi  \eta_{x_{i}x_{j}} - 4 s\lambda^{2}\xi \eta_{x_{i}} \eta_{x_{j}},\quad \Phi_{4}^{i} = -4 s^{3} \lambda^{3} \xi^{3} |\nabla\eta|^{2} \eta_{x_{i}} ,
 	\\ \ns\ds \Phi_{5} =
 	- 6 s^{3} \lambda ^{4}\xi^{3} |\nabla\eta|^{4}
 	- 12 s ^{3}\lambda ^{3}\xi^{3}  (\nabla^{2}\eta \nabla \eta \nabla \eta )\\
 	\ns\ds \quad\qquad
 	-2 s^{3} \lambda ^{3}\xi^{3} |\nabla\eta|^{2} \Delta \eta
 	- s ^{3}\lambda ^{\frac{7}{2}}\xi^{3} |\nabla\eta|^{4},
    \end{cases}  i,j=1,\cdots,n
\end{equation}
\begin{align}
	\notag
	I_{3} & = -8 s \lambda \xi \sum\limits_{i,j=1}^{n} \eta_{x_{i}x_{j}} w_{x_{i}x_{j}}dt -4\nabla  \Delta \ell \cdot \nabla w  dt + 4  (\nabla \ell \cdot \nabla \Delta \ell ) w dt
	+2 |\nabla^{2} \ell|^{2} wdt
	\\\label{eq.I3-2}
	& \quad
	-\Delta^{2} \ell wdt+|\Delta \ell|^{2} wdt+ 8 s^{3} \lambda^{3} \xi^{3}(\nabla^{2} \eta \nabla \eta \nabla \eta) wdt
	+ s ^{3}\lambda ^{\frac{7}{2}}\xi^{3} |\nabla\eta|^{4} w dt
	+ \lambda \Delta w dt
	\\[4mm] \notag
	& \quad
	- \ell_{t} \theta f_{1} dt
	- \ell_{t} \theta g_{1} dW(t)
	+ (\ell_{t}^{2}-\ell_{tt}) w dt
	- 2 \ell_{t} \hat{w} dt,
\end{align}
\begin{align*}
        V_{1} & =[ V_{1}^{1}, V_{1}^{2},\cdots, V_{1}^{n} ],
        \quad
        V_{2} = [ V_{2}^{1}, V_{2}^{2},\cdots, V_{2}^{n} ],
        \\
        V_{1}^{j} & =  \sum_{i,k=1}^n\Big[ \sum_{l=1}^n \Phi_{1}^{l} w_{x_{k}x_{k}x_{l}} w_{x_{i}x_{i}x_{j}}  -  \frac{1}{2} \sum_{l=1}^n  \Phi_{1}^{j} w_{x_{k}x_{k}x_{l}} w_{x_{i}x_{i}x_{l}}
            + \frac{1}{2}  \Psi_{2} \Phi_{1}^{j}w_{x_{i}x_{i}}w_{x_{k}x_{k}}
          \\
          &  \quad \quad \quad \ \
           + \sum_{l=1}^{n} \Psi_{3}^{ik}\Phi_{1}^{l} w_{x_{i}x_{k}}w_{x_{l}x_{j}}
           -  \frac{1}{2}  \sum_{l=1}^{n}\Psi_{3}^{ij}\Phi_{1}^{l} w_{x_{i}x_{k}}w_{x_{k}x_{l}}
            + \Phi_{4}^{k} w_{x_{i}x_{i}x_{j}} w_{x_{k}}
            - \Phi_{4}^{k} w_{x_{i}x_{j}} w_{x_{i}x_{k}}
            \\
            & \quad \quad \quad \ \
            + \frac{1}{2} \Phi_{4}^{j}w_{x_{i}x_{k}}^{2}
            +\Big ( \Psi_{2} \Phi_{4}^{k} \delta_{ij}
            - \frac{1}{2} \Psi_{2} \Phi_{4}^{j} \delta_{ik}
            - \frac{1}{2} \Psi_{6} \Phi_{1}^{j} \delta_{ik}
            + \Psi_{3}^{ij}\Phi_{4}^{k} \Big ) w_{x_{i}}w_{x_{k}}\Big],
        \\
        V_{2}^{j}  &=
        \sum_{i,l,r,m=1}^n \Theta_{1}^{ijlrm} w_{x_{i}x_{i}x_{l}}w_{x_{r}x_{m}}
        + \sum_{i,k,l,r=1}^n \Theta_{2}^{ijklr} w_{x_{i}x_{k}}w_{x_{l}x_{r}}
        + \sum_{i,k,l=1}^n \Theta_{3}^{ijkl} w_{x_{i}x_{k}}w_{x_{l}}
        \\
        & \quad
        + \sum_{i,k=1}^n \Theta_{4}^{ijk} w_{x_{i}}w_{x_{k}}
        + \sum_{i=1}^n \Theta_{5} w_{x_{i}x_{i}x_{j}}w
        + \sum_{i,k=1}^n \Theta_{6}^{ijk} w_{x_{i}x_{k}}w
        + \sum_{i=1}^n \Theta_{7}^{ij} w_{x_{i}}w
        +  \Theta_{8}^{j} w^{2}
        + \Theta_{9}^{j},
    \end{align*}
    \begin{align*}
        M_{1} & =
        8 s \lambda^{2}\xi   | \nabla \Delta w \cdot \nabla \eta  | ^{2}
        + 32 s^{3} \lambda^{4} \xi^{3}  | \nabla^{2} w  \nabla \eta \nabla \eta   |^{2}
        + 48 s^{3} \lambda^{3} \xi^{3} \nabla^{2} \eta  (\nabla^{2} w \nabla \eta )  ( \nabla^{2} w \nabla \eta  )
        \\
        & \quad
        + 16 s^{3} \lambda^{3} \xi^{3}  ( \nabla^{2} w \nabla \eta \nabla \eta  ) \sum\limits_{i,j=1}^{n} \eta_{x_{i}x_{j}} w_{x_{i}x_{j}}
        -16 s^{3} \lambda ^{4} \xi ^{3} |\nabla \eta|^{2}  | \nabla^{2} w \nabla \eta  | ^{2}
        \\
        & \quad +6 s^{3} \lambda^{4} \xi ^{3} |\nabla \eta |^{4} | \nabla^{2} w|^{2}
        + 4 s^{3} \lambda^{3} \xi ^{3}  ( \nabla^{2} \eta \nabla \eta \nabla \eta  ) |\nabla ^{2} w|^{2}
        + 2  s^{3} \lambda^{3} \xi ^{3} |\nabla \eta|^{2} \Delta \eta |\nabla^{2} w|^{2}
        \\
        & \quad+ 2 s^{3} \lambda^{4} \xi ^{3} |\nabla \eta |^{4} | \Delta w|^{2}
        - 4 s^{3} \lambda^{3} \xi ^{3}  ( \nabla^{2} \eta \nabla \eta \nabla \eta  ) |\Delta w|^{2}
        - 2  s^{3} \lambda^{3} \xi ^{3} |\nabla \eta|^{2} \Delta \eta |\Delta w|^{2}
        \\
        & \quad+ 40 s^{5} \lambda^{6} \xi^{5} |\nabla \eta|^{4} | \nabla w \cdot \nabla \eta|^{2}
        + 64 s^{5} \lambda ^{5} \xi^{5}  ( \nabla^{2} \eta \nabla \eta \nabla \eta  ) |\nabla w \cdot \nabla \eta|^{2}
        \\
        & \quad- 16 s^{5} \lambda^{6} \xi^{5} |\nabla \eta|^{6} |\nabla  w|^{2}
        + 8 s^{7} \lambda ^{8} \xi^{7} |\nabla \eta|^{8} w^{2}
        + \lambda |\nabla \Delta w|^{2}
        - s^{3} \lambda^{\frac{7}{2}} \xi^{3} |\nabla \eta|^{4}|\Delta w|^{2}
        \\
        & \quad
        + 4 s^{5} \lambda^{\frac{11}{2}} \xi^{5} |\nabla \eta|^{4}|\nabla w \cdot \nabla \eta|^{2}
        + 2 s^{5} \lambda^{\frac{11}{2}} \xi^{5} |\nabla \eta|^{6}|\nabla w|^{2}
        -  s^{7} \lambda^{\frac{15}{2}} \xi^{7} |\nabla \eta|^{8}w^{2},
    \end{align*}
    \begin{align*}
        M_{2} &=
        \biggl( -\frac{1}{2} \sum_{i=1}^{n}\Phi_{1x_{i}}^{i} + \Phi_{2} \biggr)|\nabla \hat{w}|^{2}
        + \sum_{i,j=1}^{n} \big(- \Phi_{1x_{j}}^{i}+\Phi_{3}^{ij}\big) \hat{w}_{x_{i}} \hat{w}_{x_{j}}
        \\
        &\quad
        + \frac{1}{2}  \biggl[  \sum_{i,j=1}^{n}
        \big(
            \Phi_{1x_{i}x_{j}x_{j}}^{i}
            -  \Phi_{2x_{i}x_{i}}  \delta_{ij}
            -  \Phi_{3x_{i}x_{j}}^{ij}
            +  \Phi_{4x_{i}}^{i} \delta_{ij}
        \big)
            - 2\Phi_{5}
            \biggr] \hat{w}^{2}
        \\
        & \quad
        - \sum_{i,k,j=1}^n \big( \Psi_{5}^{i} \Phi_{1}^{k} \big)_{x_{j}} w_{x_{i}} w_{x_{k}x_{j}}
        - 4 s^{2} \lambda^{3} \xi^{2} |\nabla ^{2} w \nabla \eta|
        - 2 s^{2} \lambda^{3} \xi^{2} |\nabla \eta|^{2}|\Delta w|^{2}\\
        & \quad
        +  s^{4} \lambda^{5} \xi^{4} |\nabla \eta|^{4}|\nabla w |^{2},
    \end{align*}
    \begin{align*}
    & \Lambda_{1}^{ijkl} =
    \Phi_{3x_{i}x_{j}}^{kl}
    +\Phi_{3x_{i}x_{l}}^{kj}
    +\sum_{r=1}^{n} \Big(
        \frac{1}{2} \Phi_{2x_{r}x_{r}} \delta_{ij}\delta_{kl}
        - \Phi_{3x_{j}x_{r}}^{kr} \delta_{il}
        - \Phi_{3x_{i}x_{r}}^{kr} \delta_{lj}
        +\frac{1}{2}\sum_{m=1}^{n} \Phi_{3x_{r}x_{m}}^{rm} \delta_{ik} \delta_{lj}
    \Big),
    \\
   & \Lambda_{2}^{ij}
        =
        \sum_{k=1}^n\Big [  - \Phi_{4 x_{i}x_{k}x_{k}}^{j}
     +  \frac{1}{2} \Phi_{4x_{i}x_{j}x_{k}}^{k}
     - 2\Phi_{5x_{i}x_{j}}
    + (\Psi_{2}\Phi_{3}^{jk} )_{x_{i}x_{k}}
    - \frac{1}{2}  (\Psi_{2}\Phi_{3}^{ij} )_{x_{k}x_{k}}
    \\
    &  \qquad \qquad \quad 
    - \sum_{l=1}^{n} \frac{1}{2}   (\Psi_{2}\Phi_{3}^{kl} \delta_{ij} )_{x_{k}x_{l}}
    +   (\Psi_{3}^{ik}\Phi_{2}  )_{x_{k}x_{j}}
    - \sum_{l=1}^{n} \frac{1}{2}  (\Psi_{3}^{kl}\Phi_{2} \delta_{ij} )_{x_{k}x_{l}}
    - \frac{1}{2}  (\Psi_{3}^{ij}\Phi_{2}  )_{x_{k}x_{k}}
    \\
    & \qquad \qquad \quad  
    + \sum_{l=1}^{n}  (\Psi_{3}^{ik}\Phi_{3}^{jl} )_{x_{k}x_{l}}
    - \sum_{l=1}^{n} \frac{1}{2}  (\Psi_{3}^{ij}\Phi_{3}^{kl} )_{x_{k}x_{l}}
    - \frac{1}{2}  (\Psi_{3}^{kl}\Phi_{3}^{ij} )_{x_{k}x_{l}}
    + \frac{1}{2}  (\Psi_{4}^{i}\Phi_{1}^{j} )_{x_{k}x_{k}}
    \\
    & \qquad \qquad \quad  
    +  (\Psi_{4}^{i}\Phi_{2} )_{x_{j}}
    + \frac{1}{2}  (\Psi_{4}^{k}\Phi_{2}  \delta_{ij} )_{x_{k}}
    -  (\Psi_{4}^{i}\Phi_{3}^{jk} )_{x_{k}}
    + \frac{1}{2}  (\Psi_{4}^{k}\Phi_{3}^{ij} )_{x_{k}}
    +  (\Psi_{5}^{i}\Phi_{2} )_{x_{j}}
    \\[2mm]
    & \qquad \qquad \quad   
    + \frac{1}{2}  (\Psi_{5}^{k}\Phi_{2}  \delta_{ij} )_{x_{k}}
    -  (\Psi_{5}^{i}\Phi_{3}^{jk} )_{x_{k}}
    + \frac{1}{2}  (\Psi_{5}^{k}\Phi_{3}^{ij} )_{x_{k}}
    \Big ], \\
      &   \Lambda_{3} =
        \sum_{i=1}^n \Big[  \sum_{j=1}^n \frac{1}{2} \Phi_{5x_{i}x_{i}x_{j}x_{j}}
        + \sum_{j=1}^n \frac{1}{2}   ( \Psi_{2} \Phi_{5}  )_{x_{j}x_{j}}
        + \sum_{j=1}^n \frac{1}{2} (\Psi_{3}^{ij} \Phi_{5})_{x_{i}x_{j}}
        - \frac{1}{2}   ( \Psi_{4} ^{i} \Phi_{5}  )_{x_{i}}
        - \frac{1}{2}   ( \Psi_{5} ^{i} \Phi_{5}  )_{x_{i}}
        \\
        &  \quad \quad \quad \quad
        - \sum_{j=1}^n \frac{1}{2} ( \Psi_{6} \Phi_{1}^{i}  )_{x_{i}x_{j}x_{j}}
        + \frac{1}{2}  ( \Psi_{6} \Phi_{2}  )_{x_{i}x_{i}}
         \Big]
        ,
        \\
        & \Lambda_{4} =
         \sum_{i,j =1}^{n}  \big[
                -\Phi_{1t}^{i} w_{x_{i}x_{j}x_{j}}
                + \big(
                    - \Phi_{2t} \delta_{ij}
                    - \Phi_{3t}^{ij}
                \big)w_{x_{i}x_{j}}
        \big] \hat{w} dt
        - \sum_{i=1}^{n} \Phi_{4t}^{i}
        w_{x_{i}} \hat{w} dt
        -\Phi_{5t} w \hat{w} dt
        \\
        &  \qquad
        + \sum_{i,j=1}^{n} \big(\Phi^{i}_{1x_{i}} \hat{w} + \Phi^{i}_{1} \hat{w}_{x_{i}} - \Phi_{2} \hat{w}  \delta_{ij}\big)  \big[ (\theta f_{1})_{x_{j}x_{j}} dt +  (\theta g_{1})_{x_{j}x_{j}} d W(t) \big]
        - \Phi^{i}_{5} \hat{w} \big( \theta f_{1} dt +  \theta g_{1} d W(t) \big)
        \\
        &   \qquad
        - \sum_{i,j=1}^{n} \Phi^{ij}_{3} \hat{w} \big[ (\theta f_{1})_{x_{i}x_{j}} dt +  (\theta g_{1})_{x_{i}x_{j}} d W(t) \big]
        - \sum_{i=1}^{n} \Phi^{i}_{4} \hat{w} \big[ (\theta f_{1})_{x_{i}} dt +  (\theta g_{1})_{x_{i}} d W(t) \big]
        \\
        & \qquad
        -  \sum_{i,j=1}^{n} \big(
                \Phi^{i}_{1} d w_{x_{i}x_{i}x_{j}}
            + \Phi_{2} \delta_{ij} d w_{x_{i}x_{i}}
            + \Phi_{3}^{ij} d w_{x_{i}x_{j}}
            + \Phi^{i}_{4} \delta_{ij}  d w_{x_{i}}\big)
        d \hat{w}
        + \Phi_{5} d w d \hat{w},
    \end{align*}
    and
    \begin{align*}
     &   \Theta_{1}^{ijklrm} =
       \Phi_{2} \delta_{lj}\delta_{rm}
     + \Phi_{3}^{rm}    \delta_{lj}
     - \Phi_{3}^{mj}    \delta_{lr}
     + \Phi_{3}^{lr}    \delta_{mj},                                                                                        \\
     &   \Theta_{2}^{ijklr} =
        -\frac{1}{2}\Phi_{2x_{j}} \delta_{ik}\delta_{lr}
        - \Phi_{3x_{i}}^{lr} \delta_{kj}
        + \Phi_{3x_{k}}^{rj} \delta_{il}
        + \sum_{m=1}^{n} \Phi_{3x_{m}}^{rm} \delta_{kj}\delta_{il}
        - \frac{1}{2} \sum_{m=1}^{n} \Phi_{3x_{m}}^{jm} \delta_{il} \delta_{kr}
        - \Phi_{3x_{i}}^{kl} \delta_{rj}
        ,
        \\
      &   \Theta_{3}^{ijkl} =
        - \Phi_{4x_{i}}^{l}\delta_{kj}
        - \Phi_{5}\delta_{ik}\delta_{lj}
        + \Psi_{2}\Phi_{3}^{ik}\delta_{lj}
        -\Psi_{2}\Phi_{3}^{ij}\delta_{kl}
        + \Psi_{3}\Phi_{2}\delta_{ik}
        -\Psi_{3}\Phi_{2}\delta_{ij}
        +\Psi_{3}^{ik}\Phi_{3}^{lj}
        \\[2mm]
        & \qquad\quad\;
        -\Psi_{3}^{ij}\Phi_{3}^{kl}
        +\Psi_{4}^{l}\Phi_{1}^{k}\delta_{ij}
        +\Psi_{5}^{l}\Phi_{1}^{k}\delta_{ij},
        \\
       &   \Theta_{4}^{ijk}=
        \Phi_{4x_{i}x_{j}}^{k}
         -  \frac{1}{2}\Phi_{4x_{i}x_{k}}^{j}
         + 2 \Phi_{5x_{i}}\delta_{kj}-\Phi_{5x_{j}}\delta_{ik}
         -  ( \Psi_{2}\Phi_{3}^{jk}  )_{x_{i}}
         + \frac{1}{2}  ( \Psi_{2}\Phi_{3}^{ik}  )_{x_{j}}
        \\
        & \qquad\quad 
        + \sum_{l=1}^{n} \frac{1}{2}  ( \Psi_{2}\Phi_{3}^{jl})_{x_{l}} \delta_{ik}
        - \sum_{l=1}^{n}  ( \Psi_{3}^{il}\Phi_{2}  )_{x_{l}}\delta_{kj}
        +  \frac{1}{2} \sum_{l=1}^{n}  ( \Psi_{3}^{lj}\Phi_{2}  )_{x_{l}}\delta_{ik}
         +  \frac{1}{2}   ( \Psi_{3}^{ik}\Phi_{2}  )_{x_{j}}
        \\
        & \qquad\quad 
        -  \sum_{l=1}^{n}  ( \Psi_{3}^{ij}\Phi_{3}^{kl}  )_{x_{l}}
        + \frac{1}{2}  \sum_{l=1}^{n}  ( \Psi_{3}^{ik}\Phi_{3}^{jl}  )_{x_{l}}
        + \frac{1}{2}  \sum_{l=1}^{n}  ( \Psi_{3}^{lj}\Phi_{3}^{ik}  )_{x_{l}}
        - \frac{1}{2}    ( \Psi_{4}^{i}\Phi_{1}^{k}  )_{x_{j}}
        +  \Psi_{4}^{i} \Phi_{2} \delta_{kj} 
        \\
        & \qquad\quad - \frac{1}{2}\Psi_{4}^{j} \Phi_{2} \delta_{ik}
        + \Psi_{4}^{i}\Phi_{3}^{kj}
        -\frac{1}{2} \Psi_{4}^{j}\Phi_{3}^{ik}
        +  \Psi_{5}^{i} \Phi_{2} \delta_{kj}
        - \frac{1}{2}\Psi_{5}^{j} \Phi_{2} \delta_{ik}
        + \Psi_{5}^{i}\Phi_{3}^{kj}
        -\frac{1}{2} \Psi_{5}^{j}\Phi_{3}^{ik},
    \\ \vspace{2mm}
      &   \Theta_{5}= \Phi_{5},
        \\ \vspace{2mm}
      &   \Theta_{6}^{ijk}= - \Phi_{5x_{j}}\delta_{ik}+ \Psi_{6}\Phi_{1}^{k}\delta_{ij},
        \\
      &   \Theta_{7}^{ij} = \sum_{k=1}^{n} \Phi_{5x_{k}x_{k}}\delta_{ij} + \Psi_{2} \Phi_{5} \delta_{ij} + \Psi_{3}^{ij}\Phi_{5} -  ( \Psi_{6}\Phi_{1}^{i}  )_{x_{j}} + \Psi_{6}\Phi_{2} \delta_{ij} + \Psi_{6}\Phi_{3}^{ij},
        \\
      &   \Theta_{8}^{j} = - \frac{1}{2} \sum_{i=1}^n \Phi_{5x_{i}x_{i}x_{j}} - \frac{1}{2} ( \Psi_{2}\Phi_{5}  )_{x_{j}} - \frac{1}{2} \sum_{i=1}^n  ( \Psi_{3}^{ij} \Phi_{5}  )_{x_{i}} + \frac{1}{2} \Psi_{4}^{j}\Phi_{5} + \frac{1}{2} \Psi_{5}^{j}\Phi_{5} 
      \\
      & \quad \quad \ +\frac{1}{2}\sum_{i=1}^n  ( \Psi_{6}\Phi_{1}^{j}  )_{x_{i}x_{i}}
        -\frac{1}{2}    (\Psi_{6}\Phi_{2} )_{x_{j}}
        - \frac{1}{2}\sum_{i=1}^n  ( \Psi_{6}\Phi_{3}^{ij}  )_{x_{i}}
        +\frac{1}{2} \Psi_{6} \Phi_{4}^{j}.
        \\
        & \Theta_{9}^{j}
        =
        \sum_{i=1}^{n} \biggl(
            \Phi^{i}_{1 x_{i}} \hat{w}_{x_{j}} \hat{w}
            + \Phi_{1}^{i} \hat{w}_{x_{i}} \hat{w}_{x_{j}}
            - \frac{1}{2} \Phi_{1 x_{i}x_{j}}^{i} \hat{w}^{2}
            - \frac{1}{2} \Phi_{1}^{j} \hat{w}_{x_{i}}^{2}
            - \Phi_{3}^{ij} \hat{w} \hat{w}_{x_{i}}
            + \frac{1}{2} \Phi_{3x_{i}}^{ij} \hat{w}^{2}
        \biggr)
        \\
        & \qquad ~-\Phi_{2} \hat{w} \hat{w}_{x_{j}}
        + \frac{1}{2} \Phi_{2x_{j}} \hat{w}^{2}
        - \frac{1}{2} \Phi_{4}^{j} \hat{w}^{2}.
    \end{align*}
\end{theorem}
\begin{remark} 
Since we do not put any further assumptions on $v$ and $\hat v$, the identity \eqref{eq.finalEquation-3} seems to be very complicated. For solutions to \eqref{eq.reference} or \eqref{eq.dualSCB1},  many terms, such as $V_1$ and $V_2$, will merge or vanish by means of the boundary
conditions. Furthermore, compared with energy terms, such as $8 s \lambda^{2}\xi| \nabla \Delta w \cdot \nabla \eta  | ^{2}$, $2 s^{3} \lambda^{4} \xi ^{3} |\nabla \eta |^{4} | \Delta w|^{2}$, $40 s^{5} \lambda^{6} \xi^{5} |\nabla \eta|^{4} | \nabla w \cdot \nabla \eta|^{2}$, $\lambda |\nabla \Delta w|^{2}$, many terms in \eqref{eq.finalEquation-3} is of no great importance. We only need to estimate their orders with respect to $s$ and $\lambda$. Hence, an effective way to simplify \eqref{eq.finalEquation-3} is that we do not write these terms explicitly and just  claim that the order of $s$ and $\lambda$ for them are lower than the terms with a ``good sign". However, we do not do this since we want to provide the full details for readers, and particularly for beginners. 
\end{remark} 

Although the form
of the identity \eqref{eq.finalEquation-3} is very complex, its proof follows from some basic computations. To avoid defocusing the main theme of this paper, we put the proof of Theorem \ref{thm.fundamentalIndentity-3} in the appendix.

\section{Carleman estimate for the adjoint equation}\label{sc.3}

This section is devoted to establishing a Carleman estimate for a backward stochastic plate equation. To this end, we first introduce the weight function $\eta$.

For any $ \delta >0 , T>0 $ and $ 0< \varepsilon_{1} < \frac{1}{2} $, 
we choose $ x_{0}  \in  \mathbb{R}^{n} \backslash \overline{G} $ such that 
\begin{equation} \label{eq.e2}
	R_{0} \deq \min_{x \in \overline{G}} |x - x_{0}|^{2} > 2 \delta,
\end{equation}
and choose sufficiently large $ \beta $  satisfing
\begin{equation}\label{eq.e3}
R_{1} \deq \max_{x \in \overline{G}} |x - x_{0}|^{2} \leq \beta \varepsilon_{1}^{2} T^{2} - \delta.
\end{equation}
We also choose sufficiently small $ \varepsilon_{0} $ with $ 0< \varepsilon_{0} < \varepsilon_{1}  $ such that
\begin{equation}\label{eq.e4}
	R_{0} - \beta \varepsilon_{0}^{2} T^{2} \geq  \delta.
\end{equation}

Let
\begin{equation}\label{eq.xiEta}
\eta(t,x)=|x-x_{0}|^{2} - \beta \bigg(t-\frac{T}{2}\bigg)^{2}.
\end{equation}
From \cref{eq.xiEta,eq.e2,eq.e3,eq.e4}, it is easy to see that $\eta$ satisfies the following conditions.

\smallskip

\begin{condition}
    \label{prop.eta}
    \quad\vspace{-1mm}

    \begin{enumerate}[(1).]
        \item $ |\eta(t,x)|_{C^{2}(\overline{Q})} \leq C_{1} $.

        \item $ |\nabla \eta(t,x) | \geq C_{2} > 0, \quad \forall ~ (t,x) \in  \overline{Q} $.

        \item For all $ (t,x) $ in  $ J_{1}\deq[ ( 0 , T/2 - \varepsilon_{1} T ) \cup ( T/2+ \varepsilon_{1} T, T ) ] \times G, $
        it holds that $ \eta(t,x) \leq  - \delta $.

        \item For all $ (t,x) $ in $ J_{2}\deq( T/2 - \varepsilon_{0} T , T/2 + \varepsilon_{0} T )  \times G, $
        it holds that
        $ \eta(t,x) \geq \delta $.
    \end{enumerate}
\end{condition}

Recall that $ \theta=e^{\ell} $, $ \ell=s\xi $ and $ \xi=e^{\lambda \eta} $. With $\eta$ given by \eqref{eq.xiEta}, the functions $\ell$ and $\theta$ are also defined.

We also need the following known result.
\begin{lemma}\label{lemma.ellipticE-2}\cite[Theorem 2.1]{fu2019carleman}
	Let $ q \in H^{2}_{0}(G) $.
	Then there exists a constant $C>0$ independent of $ s $ and $\lambda$, and parameter $\widehat{\lambda}>1$ and $ \widehat{s} > 1 $ such that, for all $\lambda \geq \widehat{\lambda}$ and $ s \geq \widehat{s}$,
	\begin{align}
		\notag
		s^{4}  \lambda^{6}  \int_{Q}  \xi^{4}  \theta^{2} \big(
		s^{2}  \lambda^{2} \xi^{2}|q|^{2}
		+| \nabla q|^{2}
		\big)dx dt
		\leq
		C   \int_{Q} s^{3}  \lambda^{4} \xi^{3} \theta^{2} |\Delta q|^{2} dx dt.
	\end{align}
\end{lemma}

We have the following Carleman estimate.

\begin{theorem} \label{thm.ceII}
    There exist constants $ C>0$ and $\lambda_{0} >0 $ such that for all $\lambda   \geq \lambda_{0}$, one can find $ s_{0}= s_{0}(\lambda) > 0 $ so that for any $ s \geq s_{0} $,
    \begin{equation*}
    (v, \hat{v})  \in L^{2}_{\mathbb{F}} ( \Omega; C([0,T]; H^{4}(G) \cap H_{0}^{2}(G)) ) \times  L^{2}_{\mathbb{F}} ( \Omega; C([0,T];  H_{0}^{2}(G)) )
    ,
    \end{equation*}
    and
    \begin{equation*}
    f_{1}, f_{2}, g_{2} \in L^{2}_{\mathbb{F}} (0,T;H_{0}^{2}(G)), \quad
    g_{1} \in L^{2}_{\mathbb{F}} (0,T;H^{4}(G) \cap H_{0}^{2}(G))
    \end{equation*}
    satisfying
    \begin{equation}\label{eq.carlemanEQII}
    \left\{
    \begin{array}{ll}
    \ds d v =(\hat{v} + f_{1}) d t+ g_{1} d W(t) & \text { in } Q, \\
    \ns\ds d \hat{v} +\Delta^2 v d t=f_{2} d t+ g_{2} d W(t) & \text { in } Q, \\
    \ns\ds v= \frac{\partial v}{ \partial \nu}=0 & \text { on } \Sigma,
    \end{array}
    \right.
    \end{equation}
    and
    \begin{equation}\label{eq.hT}
    v(0,\cdot)=v(T,\cdot)=\hat{v}(0,\cdot)=\hat{v}(T,\cdot)=0 \quad  \mbox{ in } G, \quad \text{$ \mathbb{P} $-\rm a.s.},
    \end{equation}
    it holds that
    \begin{align}\label{eq.08012022}
    \notag
    &
    \mathbb{E} \int_{Q} \theta^{2} \bigl(
    s \lambda \xi |\nabla \hat{v}|^{2}
    + s^{3} \lambda^{\frac{7}{2}} \xi^{3} | \hat{v}|^{2}
    +  \lambda |\nabla \Delta v|^{2}
    + s^{2} \lambda^{4} \xi^{2} |\nabla^{2} v|^{2}
    + s^{3} \lambda^{4} \xi^{3} |\Delta  v|^{2}
    \\\notag
    & \qquad \quad
    + s^{4} \lambda^{6} \xi^{4} |\nabla  v|^{2}
    + s^{6} \lambda^{8} \xi^{6} |v|^{2}
    \bigr) dxdt
    \\
    &
    \leq
    C \mathbb{E} \int_{Q} \theta^{2} \big(
    s^{6} \lambda^{6} \xi^{6} f_{1}^{2}
    + s^{4} \lambda^{4} \xi^{4} |\nabla f_{1}|^{2}
    + s^{2} \lambda^{2} \xi^{2} |\nabla^{2} f_{1}|^{2}
    + f_{2}^{2}
    +s^{6} \lambda^{6} \xi^{6} g_{1}^{2}
    \\ \notag
    & \qquad \qquad \quad ~
    +s^{4} \lambda^{4} \xi^{4} |\nabla g_{1}|^{2}
    +s^{2} \lambda^{2} \xi^{2} |\nabla^{2} g_{1}|^{2}
    +|\nabla \Delta g_{1}|^{2}
    +s^{2} \lambda^{2} \xi^{2} |g_{2}|^{2}
    \big) dxdt
    \\ \notag
    & \quad
    + C \mathbb{E} \int_{\Sigma} \theta^{2} \bigr(
    s \lambda \xi |\nabla \Delta v|^{2}
    + s^{3} \lambda^{3} \xi^{3} |\Delta v|^{2}
    \bigl)d\Gamma dt.
    \end{align}
\end{theorem}
\begin{proof}[Proof of \cref{thm.ceII}]

    In what follows, for a positive integer $ r $, we denote by $ O(\lambda^{r}) $ a function of order $ \lambda^{r} $ for large $ \lambda $. Similarly, we use the notation $ O(e^{C\lambda}) $.
    
    \smallskip

    In order to shorten the formulae, we define
    \begin{align}\label{eq.AB-2}
        \notag
        &\mathcal{A}_{1}\deq  \mathbb{E}\int_{Q} \big(s^{5} O(e^{C\lambda})+s^{6}  \xi^{6} O(\lambda^{7})\big)|w|^{2} d x d t ,\\ \notag
        &  
         \mathcal{A}_{2} \deq \mathbb{E}\int_{Q} \big(s^{3} O(e^{C\lambda})+s^{4}  \xi^{4} O(\lambda^{5})\big)|\nabla w|^{2} d x d t,
        \\ \notag
        &\mathcal{A}_{3} \deq \mathbb{E}\int_{Q} \big(s O(e^{C\lambda})+s^{2}  \xi^{2} O(\lambda^{3})\big)|\nabla^{2} w|^{2} d x d t,\\  
        &\mathcal{A}_{4} \deq \mathbb{E}\int_{Q} O(1) |\nabla \Delta w|^{2} d x d t,\\ \notag &
        \mathcal{B}_{2} \deq \mathbb{E}\int_{\Sigma} \big(s^{2} O(e^{C\lambda})+s^{3}  \xi^{3} O(\lambda^{2})\big)|\nabla^{2} w|^{2} d x d t, \\ \notag
        & \hat{\mathcal{A}}_{1} \deq \mathbb{E}\int_{Q} \big(s^{2} O(e^{C\lambda})+s^{3}  \xi^{3} O(\lambda^{3})\big)|\hat{w}|^{2} d x d t,\\ \notag &
        \hat{\mathcal{A}}_{2} \deq  \mathbb{E}\int_{Q} \big(O(\lambda)+s  \xi O(1)\big)|\nabla \hat{w}|^{2} d x d t,
    \end{align}
    and
    \begin{equation*}
        \mathcal{A}\deq\mathcal{A}_{1}+\mathcal{A}_{2}+\mathcal{A}_{3}+\mathcal{A}_{4}
        +\hat{\mathcal{A}}_{1}
        +\hat{\mathcal{A}}_{2}.
    \end{equation*}

Integrating the equality \cref{eq.finalEquation-3} on $ Q $, taking mathematical expectation in both sides, and noting \cref{eq.hT}, we obtain
\begin{align}
    \notag
        & 2 \mathbb{E} \int_{Q} \theta I_{2} (  d \hat{v} + \Delta^{2} v dt ) dx
        - 2 \mathbb{E} \int_{Q} \operatorname{div} (V_{1}+V_{2}) dxdt
        \\[1mm] \label{11.25-eq5}
        & =
        2 \mathbb{E} \int_{Q}  I_{2}^{2} dxdt
        + 2 \mathbb{E} \int_{Q} I_{2}I_{3} dx
        + 2 \mathbb{E} \int_{Q} (M_{1} + M_{2}) dx dt
        \\ \notag
        & \quad 
        + 2 \sum_{i,j,k,l=1}^n \mathbb{E} \int_{Q} \Lambda^{ijkl}_{1} w_{x_{i}x_{j}}w_{x_{k}x_{l}}dx dt
        + 2 \sum_{i, j=1}^n \mathbb{E} \int_{Q} \Lambda_{2}^{ij} w_{x_{i}} w_{x_{j}}dx dt\\ \notag
        & \quad 
        + 2 \mathbb{E} \int_{Q} \Lambda_{3} w^{2} dx dt
        + 2 \mathbb{E} \int_{Q} \Lambda_{4} dx.
\end{align}
By \cref{prop.eta}, \cref{eq.Ic-2,eq.I2c-2,eq.AB-2}, we have
\begin{equation}\label{11.25-eq1}
\sum_{i,j,k,l=1}^n \mathbb{E} \int_{Q} \Lambda^{ijkl}_{1} w_{x_{i}x_{j}}w_{x_{k}x_{l}}dx dt
\geq
- C \mathbb{E}\int_{Q} s \lambda ^{4} \xi |\nabla^{2} w|^{2} d x d t
\geq
\mathcal{A}_{3},
\end{equation}
\begin{equation}\label{11.25-eq2}
\sum_{i, j=1}^n \mathbb{E} \int_{Q} \Lambda_{2}^{ij} w_{x_{i}} w_{x_{j}}dx dt
\geq
- C \mathbb{E}\int_{Q} \big(s^{3} \lambda ^{6} \xi^{3} +s^{2} \lambda ^{6} \xi^{2} \big) |\nabla^{2} w|^{2} d x d t
\geq
\mathcal{A}_{2}
,	
\end{equation}
\vspace{2mm}
\begin{equation}\label{11.25-eq3}
\mathbb{E} \int_{Q} \Lambda_{3} w^{2} dx dt
\geq
- C \mathbb{E}\int_{Q} s^{3} \lambda ^{8} \xi^{3} \big( 1 + s \xi +s^{2}  \xi^{2} \big) |\nabla w|^{2} d x d t
\geq
\mathcal{A}_{1}	
\end{equation}
and
\begin{align}\notag
    \mathbb{E} \int_{Q} \Lambda_{4} dx
    &\geq
    - C\mathbb{E} \int_{Q} \big[
        |\nabla \Delta w|^{2}
        + s O(e^{C \lambda}) |\nabla^{2} w|^{2}
        + s^{3} O(e^{C \lambda}) |\nabla w|^{2}
        + s^{5} O(e^{C \lambda}) | w|^{2}
    \\ \notag
    & \qquad \qquad \quad ~
        + s^{2} O(e^{C \lambda}) | \hat{w}|^{2}
        + s^{3} \lambda^{3} \xi^{3} |\hat{w}|^{2}
        + (s \xi + \lambda) |\nabla \hat{w}|^{2}
    \big] dxdt
    \\[1mm]
    & \quad \label{11.25-eq4}
    - C \mathbb{E} \int_{Q} \theta^{2} \big(
        s^{6} \lambda^{6} \xi^{6} f_{1}^{2}
        + s^{4} \lambda^{4} \xi^{4} |\nabla f_{1}|^{2}
        + s^{2} \lambda^{2} \xi^{2} |\nabla^{2} f_{1}|^{2}
        +s^{6} \lambda^{6} \xi^{6} g_{1}^{2}
        \\ \notag
        & \qquad \qquad \quad ~
        +s^{4} \lambda^{4} \xi^{4} |\nabla g_{1}|^{2}
        +s^{2} \lambda^{2} \xi^{2} |\nabla^{2} g_{1}|^{2}
        +|\nabla \Delta g_{1}|^{2}
        +s^{2} \lambda^{2} \xi^{2} |g_{2}|^{2}
        \big) dxdt
    \\[1mm] \notag
    & \geq
    - C \mathbb{E} \int_{Q} \theta^{2} \big(
        s^{6} \lambda^{6} \xi^{6} f_{1}^{2}
        + s^{4} \lambda^{4} \xi^{4} |\nabla f_{1}|^{2}
        + s^{2} \lambda^{2} \xi^{2} |\nabla^{2} f_{1}|^{2}
        +s^{6} \lambda^{6} \xi^{6} g_{1}^{2}
        \\ \notag
        & \qquad \qquad \quad ~
        +s^{4} \lambda^{4} \xi^{4} |\nabla g_{1}|^{2}
        +s^{2} \lambda^{2} \xi^{2} |\nabla^{2} g_{1}|^{2}
        +|\nabla \Delta g_{1}|^{2}
        +s^{2} \lambda^{2} \xi^{2} |g_{2}|^{2}
        \big) dxdt
    + \mathcal{A}.
\end{align}
It follows from \eqref{11.25-eq5}--\eqref{11.25-eq4} that
\begin{align}
    \notag
    & 2 \mathbb{E} \int_{Q} \theta I_{2} (  d\hat{v} + \Delta^{2} v dt ) dx
    - 2 \mathbb{E} \int_{Q} \operatorname{div} (V_{1}+V_{2}) dxdt
    \\ \label{eq.beforeS1-2}
    & \geq
    2 \mathbb{E} \int_{Q}  I_{2}^{2} dxdt
    + 2 \mathbb{E} \int_{Q} I_{2}I_{3} dx
    + 2 \mathbb{E} \int_{Q} (M_{1}+M_{2}) dx dt
    \\ \notag
    & \qquad
    - C \mathbb{E} \int_{Q} \theta^{2} \big(
        s^{6} \lambda^{6} \xi^{6} f_{1}^{2}
        + s^{4} \lambda^{4} \xi^{4} |\nabla f_{1}|^{2}
        + s^{2} \lambda^{2} \xi^{2} |\nabla^{2} f_{1}|^{2}
        +s^{6} \lambda^{6} \xi^{6} g_{1}^{2}
        +s^{4} \lambda^{4} \xi^{4} |\nabla g_{1}|^{2}
        \\[1mm] \notag
        & \qquad \qquad  \qquad \quad ~
        +s^{2} \lambda^{2} \xi^{2} |\nabla^{2} g_{1}|^{2}
        +|\nabla \Delta g_{1}|^{2}
        +s^{2} \lambda^{2} \xi^{2} |g_{2}|^{2}
        \big) dxdt
    + \mathcal{A}
    .
\end{align}

We will estimate the terms in \cref{eq.beforeS1-2} one by one. The procedure is divided into three steps.

\medskip

\emph{Step 1.} In this step, we consider the divergence terms.

Thanks to the boundary conditions satisfied by $v$, it is easy to check that
\begin{align}\label{eq.BC-2}
    w & =0, \quad \nabla w = 0,  \quad \nabla^{2} w  = \theta \nabla^{2} v \quad\mbox{ on } \Sigma.
\end{align}
We also have
\begin{align}\label{eq.b2-2}
    \frac{\partial w_{x_{i}}}{\partial \nu} \nu^{j}
    = w_{x_{i}x_{j}}
    = w_{x_{j}x_{i}}
    = \frac{\partial w_{x_{j}}}{\partial \nu} \nu^{i},
\end{align}
which implies 
\begin{align}\notag
    |\Delta w|^{2} &
    =\sum_{i,j=1}^{n} w_{x_{i}x_{i}} w_{x_{j}x_{j}}
    =\sum_{i,j=1}^{n} \frac{\partial w_{x_{i}}}{\partial \nu} \nu^{i} \frac{\partial w_{x_{j}}}{\partial \nu} \nu^{j}\\ \label{eq.b21-2}
    &
    =\sum_{i,j=1}^{n} \frac{\partial w_{x_{i}}}{\partial \nu} \nu^{j} \frac{\partial w_{x_{j}}}{\partial \nu} \nu^{i}
    = \sum_{i,j=1}^{n} w_{x_{i}x_{j}}^{2}
    \\ \notag
    & =
    |\nabla^{2} w|^{2}
    = \sum_{i=1}^{n} \biggl( \frac{\partial w_{i}}{\partial \nu}\biggr)^{2} \quad\mbox{ on } \Sigma.
\end{align}

Thanks to \cref{eq.BC-2},   we obtain that
\begin{align}\notag
    V_{1} \cdot \nu
    & = \sum_{i,j,k,l=1}^{n} \Phi_{1}^{l} w _{x_{k}x_{k}x_{l}}w_{x_{i}x_{i}x_{j}} \nu^{j}
    - \frac{1}{2} \sum_{i,j,k,l=1}^{n} \Phi_{1}^{j}  w _{x_{k}x_{k}x_{l}}w_{x_{i}x_{i}x_{l}}\nu^{j}
    \\ \label{eq.be1-2}
    & \quad + \frac{1}{2} \sum_{i,j,k=1}^{n}    \Psi_{2} \Phi_{1}^{j}w_{x_{i}x_{i}}w_{x_{k}x_{k}}  \nu^{j} + \sum_{i,j,k,l=1}^{n} \Psi_{3}^{ik}\Phi_{1}^{l} w_{x_{i}x_{k}}w_{x_{l}x_{j}}\nu^{j}
    \\\notag
    & \quad 
    - \frac{1}{2} \sum_{i,j,k,l=1}^{n} \Psi_{3}^{ij}\Phi_{1}^{l} w_{x_{i}x_{k}}w_{x_{k}x_{l}}\nu^{j}
    -\sum_{i,j,k=1}^{n} \Phi_{4}^{k} w_{x_{i}x_{j}} w_{x_{i}x_{k}}\nu^{j}
    \\ \notag
    & \quad
    + \frac{1}{2} \sum_{i,j,k=1}^{n} \Phi_{4}^{j}w_{x_{i}x_{k}}^{2}\nu^{j} \qquad\mbox{ on } \Sigma.
\end{align}
By \cref{prop.eta,eq.I2c-2}, we have that
\begin{align}\label{eq.be2-2}
    \sum_{i,j,k,l=1}^{n} \Phi_{1}^{l} w _{x_{k}x_{k}x_{l}}w_{x_{i}x_{i}x_{j}} \nu^{j}
    - \frac{1}{2} \sum_{i,j,k,l=1}^{n} \Phi_{1}^{j} \nu^{j} w _{x_{k}x_{k}x_{l}}w_{x_{i}x_{i}x_{l}}
    \geq
    -C s \lambda \xi |\nabla \Delta w|^{2} \quad\mbox{ on } \Sigma.
\end{align}
From \cref{eq.Ic-2,eq.I2c-2}, we get that
\begin{align}\label{eq.be3-2}
    \frac{1}{2} \sum_{i,j,k=1}^{n}    \Psi_{2} \Phi_{1}^{j} \nu^{j} w_{x_{i}x_{i}}w_{x_{k}x_{k}}
    =
    - 4 s^{3} \lambda ^{3} \xi^{3} |\nabla \eta|^{2} \frac{\partial \eta}{\partial \nu} |\Delta w|^{2} \quad\mbox{ on } \Sigma.
\end{align}
Combining \cref{eq.Ic-2}, \cref{eq.I2c-2}, \cref{eq.b2-2} and \cref{eq.b21-2},  we find  that
\begin{align}\notag
    \sum_{i,j,k,l=1}^{n} \Psi_{3}^{ik}\Phi_{1}^{l} w_{x_{i}x_{k}}w_{x_{l}x_{j}} \nu^{j}
    & = -16 \sum_{i,j,k,l=1}^{n} s^{3} \lambda^{3} \xi^{3} \eta_{x_{i}} \eta_{x_{k}} \eta_{x_{l}} \nu^{j} w_{x_{i}x_{k}}w_{x_{l}x_{j}}
    \\ \label{eq.be4-2}
    & = -16 \sum_{i,j,k,l=1}^{n} s^{3} \lambda^{3} \xi^{3} \eta_{x_{i}} \eta_{x_{k}}  \nu^{j} w_{x_{i}x_{k}} \frac{\partial w_{j}}{\partial \nu} \frac{\partial\eta}{\partial \nu}
    \\\notag
    & = -16 \sum_{i,j,k,l=1}^{n} s^{3} \lambda^{3} \xi^{3}  \biggl(\frac{\partial\eta}{\partial \nu}\biggr)^{3}  \biggl(\frac{\partial w_{j}}{\partial \nu}\biggr)^{2}
    \\ \notag
    & = -16 \sum_{i,j,k,l=1}^{n} s^{3} \lambda^{3} \xi^{3}  \biggl(\frac{\partial\eta}{\partial \nu}\biggr)^{3}  |\Delta w|^{2}  \quad\mbox{ on } \Sigma 
\end{align}
and
\begin{align}\notag
    &- \frac{1}{2} \sum_{i,j,k,l=1}^{n} \Psi_{3}^{ij}\Phi_{1}^{l} w_{x_{i}x_{k}}w_{x_{k}x_{l}}\nu^{j}
    -\sum_{i,j,k=1}^{n} \Phi_{4}^{k} w_{x_{i}x_{j}} w_{x_{i}x_{k}}\nu^{j}
    + \frac{1}{2} \sum_{i,j,k=1}^{n} \Phi_{4}^{j}w_{x_{i}x_{k}}^{2}\nu^{j}
    \\\label{eq.be5-2}
    & =
    8 \sum_{i,j,k,l=1}^{n} s^{3} \lambda^{3} \xi^{3}  \biggl(\frac{\partial\eta}{\partial \nu}\biggr)^{3}  |\Delta w|^{2}
    + 2 s^{3} \lambda ^{3} \xi^{3} |\nabla \eta|^{2} \frac{\partial \eta}{\partial \nu} |\Delta w|^{2} \quad\mbox{ on } \Sigma.
\end{align}
Hence, combining \cref{prop.eta,eq.be1-2,eq.be2-2,eq.be3-2,eq.be4-2,eq.be5-2},   we obtain that
\begin{align*}
    V_{1} \cdot \nu
    & \geq
    - C s \lambda \xi |\nabla \Delta w|^{2}
    - 8 \sum_{i,j,k,l=1}^{n} s^{3} \lambda^{3} \xi^{3}  \biggl(\frac{\partial\eta}{\partial \nu}\biggr)^{3}  |\Delta w|^{2}
    - 2 s^{3} \lambda ^{3} \xi^{3} |\nabla \eta|^{2} \frac{\partial \eta}{\partial \nu} |\Delta w|^{2}
    \\
    & \geq
    - C s \lambda \xi |\nabla \Delta w|^{2}
    - C s^{3} \lambda^{3} \xi^{3} |\Delta w|^{2} \quad\mbox{ on } \Sigma.
\end{align*}
This implies that
\begin{align}\label{eq.bV1-2}
    \mathbb{E} \int_{Q} \operatorname{div} V_{1} dxdt
    & \geq
    - C \mathbb{E} \int_{\Sigma} s^{3} \lambda^{3} \xi^{3} |\Delta w|^{2}  d\Gamma dt
    - C \mathbb{E} \int_{\Sigma} s \lambda \xi |\nabla \Delta w|^{2} d\Gamma dt
    .
\end{align}

Thanks to \cref{prop.eta}, \cref{eq.BC-2,eq.I2c-2,eq.AB-2}, we have
\begin{align}
    \notag
    &\mathbb{E}  \int_{Q} \operatorname{div} V_{2} dx dt
    =\mathbb{E}\int_{\Sigma}  V_{2} \cdot \nu d\Gamma dt
    \\ \notag
    &=
    \mathbb{E}\int_{\Sigma} \sum_{j=1}^{n} \Big( \sum_{i,k,l,r,m=1}^{n} \Theta_{1}^{ijklrm} w_{x_{i}x_{k}x_{l}}w_{x_{r}x_{m}}
    + \sum_{i,k,l,r=1} \Theta_{2}^{ijklr} w_{x_{i}x_{k}}w_{x_{l}x_{r}}
    +  \Theta_{9}^{j} \Big) \nu^{j} d\Gamma dt
    \\ \label{eq.bV2-2}
    & \geq  \mathbb{E}\int_{\Sigma} \Big[
        s\xi O(\lambda)|\nabla \Delta w|^{2}
    + s O(e^{C\lambda}) |\nabla^{2} w|^{2}
    \Big] d\Gamma dt
    \\ \notag
    & \geq
    - C \mathbb{E} \int_{\Sigma} s \lambda \xi |\nabla \Delta w|^{2} d\Gamma dt
    + \mathcal{B}_{2}
    .
\end{align}

Combining \cref{eq.beforeS1-2,eq.bV1-2,eq.bV2-2}, we obtain that
\begin{align}\label{eq.s1-2}
    \notag
    & 2 \mathbb{E} \int_{Q} \theta I_{2} \big(  d\hat{v} + \Delta^{2} v dt \big) dx
    + C \mathbb{E} \int_{\Sigma} s^{3} \lambda^{3} \xi^{3} |\Delta w|^{2}  d\Gamma dt
    \\ \notag
    &
    + C \mathbb{E} \int_{Q} \theta^{2} \big(
        s^{6} \lambda^{6} \xi^{6} f_{1}^{2}
        + s^{4} \lambda^{4} \xi^{4} |\nabla f_{1}|^{2}
        + s^{2} \lambda^{2} \xi^{2} |\nabla^{2} f_{1}|^{2}
        +s^{6} \lambda^{6} \xi^{6} g_{1}^{2}
        +s^{4} \lambda^{4} \xi^{4} |\nabla g_{1}|^{2} 
    \\[2mm]
    & \qquad \qquad \quad ~ +s^{2} \lambda^{2} \xi^{2} |\nabla^{2} g_{1}|^{2}
    +|\nabla \Delta g_{1}|^{2}
    +s^{2} \lambda^{2} \xi^{2} |g_{2}|^{2}
    \big) dxdt \\ \notag
    &
    + C \mathbb{E} \int_{\Sigma} s \lambda \xi |\nabla \Delta w|^{2} d\Gamma dt
    + \mathcal{A}
    + \mathcal{B}_{2}
    \\ \notag
    & \geq
    2 \mathbb{E} \int_{Q}  I_{2}^{2} dxdt
    + 2 \mathbb{E} \int_{Q} I_{2}I_{3} dx
    + 2 \mathbb{E} \int_{Q} (M_{1}+M_{2}) dx dt.
\end{align}

\medskip

\emph{Step 2.} In this step, we study $\ds 2 \mathbb{E} \int_{Q} (M_{1}+M_{2}) dxdt $ via integration by parts.

From \cref{prop.eta}, \cref{eq.AB-2} and \cref{eq.BC-2},  we get that
\begin{align} \label{eq.s2-1-2}
\notag &
-\mathbb{E}\int_{Q}  16 s^{3} \lambda ^{4} \xi ^{3} |\nabla \eta|^{2}  | \nabla^{2} w \nabla \eta  | ^{2} dxdt \\ \notag
& = - \mathbb{E}\int_{Q}  16 s^{3} \lambda ^{4} \xi ^{3} \sum_{i,j,k,l=1}^{n} \eta_{x_{i}}^{2} \eta_{x_{k}} \eta_{x_{l}} w_{x_{k}x_{j}} w_{x_{l}x_{j}} dxdt\\ 
& =  \mathbb{E} \int_{\Sigma}  8  s^{3} \lambda ^{4}  \sum_{i,j,k,l=1}^{n}  
    [ -2 \xi ^{3} \eta_{x_{i}}^{2} \eta_{x_{k}} \eta_{x_{l}} w_{x_{k}x_{j}}w_{x_{l}}
+     ( \xi ^{3} \eta_{x_{i}}^{2} \eta_{x_{k}} \eta_{x_{l}} )_{x_{j}} w_{x_{k}}w_{x_{l}}
]
    \nu^{j} d \Gamma dt
\\ \notag
& \quad  
    +   \mathbb{E}\int_{Q}  8  s^{3} \lambda ^{4} \sum_{i,j,k,l=1}^{n} [ 2 \xi ^{3} \eta_{x_{i}}^{2} \eta_{x_{k}} \eta_{x_{l}} w_{x_{k}x_{j}x_{j}} w_{x_{l}} - ( \xi ^{3}\eta_{x_{i}}^{2} \eta_{x_{k}} \eta_{x_{l}} )_{x_{j}x_{j}} w_{x_{k}}w_{x_{l}} ]dxdt
    \\ \notag
&\geq
-4 \mathbb{E}\int_{Q}  s \lambda^{2} \xi |\nabla \Delta w \nabla \eta|^{2} dxdt
-16 \mathbb{E}\int_{Q}  s^{5} \lambda^{6} \xi^{5} |\nabla \eta|^{4} |\nabla  w \nabla \eta|^{2} dxdt
- \mathcal{A}_{2}.
\end{align}
Thanks to \cref{eq.xiEta,prop.eta}, we know that there esists $ \lambda_{1} > 0 $ such that for all  $ \lambda \geq \lambda_{1} $, it holds that
\begin{align}\notag
  &  64 \mathbb{E} \int_{Q}   s^{5} \lambda ^{5} \xi^{5}  ( \nabla^{2} \eta \nabla \eta \nabla \eta  ) |\nabla w \cdot \nabla \eta|^{2} dxdt \\ \label{eq.s2-2-2}
    & =
    128 \mathbb{E} \int_{Q}   s^{5} \lambda ^{5} \xi^{5} |\nabla \eta|^{2} |\nabla w \cdot \nabla \eta|^{2} dxdt
    \\ \notag
    & \geq
    - \mathbb{E} \int_{Q}   s^{5} \lambda ^{6} \xi^{5} |\nabla \eta|^{4} |\nabla w \cdot \nabla \eta|^{2} dxdt
    .
\end{align}
Combining \cref{prop.eta}, \cref{eq.xiEta}, \cref{eq.AB-2} and  \cref{eq.BC-2}, we get that
\begin{align}\label{eq.s2-3-2} \notag
    & \mathbb{E}  \int_{Q} \big[4 s^{3} \lambda^{3} \xi ^{3}  ( \nabla^{2} \eta \nabla \eta \nabla \eta  ) +2  s^{3} \lambda^{3} \xi ^{3} |\nabla \eta|^{2} \Delta \eta\big] ( |\nabla ^{2} w|^{2}-|\Delta w|^{2} )
    \\ \notag
    & = \frac{1}{2} \mathbb{E}  \int_{Q} \Delta  \big[4 s^{3} \lambda^{3} \xi ^{3}  ( \nabla^{2} \eta \nabla \eta \nabla \eta  ) +2  s^{3} \lambda^{3} \xi ^{3} |\nabla \eta|^{2} \Delta \eta\big] |\nabla w|^{2} dx dt 
    \\
    & \quad 
    +  \mathbb{E}  \int_{Q} \nabla   \big[4 s^{3} \lambda^{3} \xi ^{3}  ( \nabla^{2} \eta \nabla \eta \nabla \eta  ) +2  s^{3} \lambda^{3} \xi ^{3} |\nabla \eta|^{2} \Delta \eta\big] \cdot \nabla w \Delta w dx dt 
    \\ \notag
    & \geq
    - \mathbb{E}  \int_{Q} s^{2}  \xi^{2} O( \lambda^{3} ) |\nabla^{2} w|^{2} dxdt -  \mathbb{E}  \int_{Q}\big(s^{3} O(e^{C\lambda})+s^{4}  \xi^{4} O(\lambda^{5})\big) |\nabla w|^{2} dxdt 
    \\ \notag
    & \geq - \mathcal{A}_{2} - \mathcal{A}_{3},
\end{align}
and that
\begin{align}\label{eq.s2-4-2}
    \notag
    & 32 \mathbb{E}  \int_{Q}
    s^{3} \lambda^{3} \xi^{3}  ( \nabla^{2} w \nabla \eta \nabla \eta  ) \sum\limits_{i,j=1}^{n} \eta_{x_{i}x_{j}} w_{x_{i}x_{j}} dxdt
    \\ \notag
    &= 32 \sum_{i,j,k,l=1}^{n} \mathbb{E}  \int_{Q}  s^{3} \lambda^{3} \xi^{3} \eta_{x_{k}x_{l}} \eta_{x_{i}}\eta_{x_{j}} w_{x_{i}x_{j}}w_{x_{k}x_{l}}  dxdt
        \\ \notag
        &= 32 \sum_{i,j,k,l=1}^{n}  \mathbb{E} \int_{\Sigma}  s^{3} \lambda^{3} \xi^{3} \eta_{x_{k}x_{l}} \eta_{x_{i}}\eta_{x_{j}}  ( w_{x_{i}} w_{x_{k}x_{l}} \nu^{j}   -  w_{x_{i}} w_{x_{j}x_{l}}\nu^{k} )d \Gamma dt \\ 
        &\quad 
        - 32 \sum_{i,j,k,l=1}^{n} \mathbb{E}  \int_{Q}  s^{3} \lambda^{3} (\xi^{3} \eta_{x_{k}x_{l}} \eta_{x_{i}}\eta_{x_{j}})_{x_{j}} w_{x_{i}} w_{x_{k}x_{l}} dxdt\\ 
        \notag
        &\quad 
        + 32 \sum_{i,j,k,l=1}^{n} \mathbb{E}  \int_{Q}  s^{3} \lambda^{3} (\xi^{3} \eta_{x_{k}x_{l}} \eta_{x_{i}}\eta_{x_{j}})_{x_{k}} w_{x_{i}} w_{x_{j}x_{l}} dxdt \\[2mm] \notag
        &\quad 
        +32 \mathbb{E} \int_{Q}   s^{3} \lambda^{3} \xi^{3} \nabla^{2} \eta  (\nabla^{2} w \nabla \eta )  ( \nabla^{2} w \nabla \eta  ) dxdt
        \\[2mm] \notag
    &
    \geq  32 \mathbb{E} \int_{Q}   s^{3} \lambda^{3} \xi^{3} \nabla^{2} \eta  (\nabla^{2} w \nabla \eta )  ( \nabla^{2} w \nabla \eta  ) dxdt -  \mathcal{A}_{2}-\mathcal{A}_{3}
    \\[2mm] \notag
    &
    =  64 \mathbb{E} \int_{Q}   s^{3} \lambda^{3} \xi^{3}  |\nabla^{2} w \nabla \eta|^{2} dxdt
    -  \mathcal{A}_{2}-\mathcal{A}_{3}
    .
\end{align}
From \cref{eq.AB-2,eq.BC-2}, we see that
\begin{align}\notag
    &\mathbb{E} \int_{Q} \bigl[
        \big( 8 s^{3} \lambda^{4} \xi^{3}  - s^{3} \lambda^{\frac{7}{2}} \xi^{3}\big) |\nabla \eta|^{4} |\Delta w|^{2}
        - 2 \big( 8 s^{3} \lambda^{4} \xi^{3}  - s^{3} \lambda^{\frac{7}{2}} \xi^{3}\big) |\nabla \eta|^{6} s^{2} \lambda^{2} \xi^{2} |\nabla w|^{2}
        \\\notag
        & \qquad ~
        + \big( 8 s^{3} \lambda^{4} \xi^{3}  - s^{3} \lambda^{\frac{7}{2}} \xi^{3}\big) |\nabla \eta|^{8} s^{4} \lambda^{4} \xi^{4} |w|^{2}
    \bigr] dxdt
    \\[1mm] \label{eq.s2-5-2}
    &= \mathbb{E} \int_{Q} \bigl[
        \big( 8 s^{3} \lambda^{4} \xi^{3}  - s^{3} \lambda^{\frac{7}{2}} \xi^{3}) |\nabla \eta|^{4} ( \Delta w +s^{2} \lambda^{2} \xi^{2} |\nabla \eta|^{2} w \big)^{2}
        \\\notag
        & \qquad \qquad
        - 2 \big( 8 s^{3} \lambda^{4} \xi^{3}  - s^{3} \lambda^{\frac{7}{2}} \xi^{3}\big) |\nabla \eta|^{6} s^{2} \lambda^{2} \xi^{2} (|\nabla w|^{2} + w \Delta w)
    \bigr] dxdt
    \\[1mm] \notag
    & \geq \mathbb{E} \int_{Q}
        \big( 8 s^{3} \lambda^{4} \xi^{3}  - s^{3} \lambda^{\frac{7}{2}} \xi^{3}\big) |\nabla \eta|^{4} \big(\Delta w +s^{2} \lambda^{2} \xi^{2} |\nabla \eta|^{2} w \big)^{2}
    dxdt
    - \mathcal{A}_{1}.
\end{align}
Noting that $ ( a + b + c )^{2}  \leq 3 (a^{2} + b^{2} + c^{2}) $ for $a, b, c\in \mathbb{R}$, we find that
\begin{align}\notag
    &\mathbb{E}  \int_{Q}  s^{3} \lambda^{4} \xi ^{3} \theta^{2} |\nabla \eta |^{4} | \Delta v|^{2}  dxdt
    \\\label{eq.s2-6-2}
    &= \mathbb{E}  \int_{Q}  s^{3} \lambda^{4} \xi^{3} |\nabla \eta|^{4} \big(
        \Delta w
        -2 s \lambda \xi \nabla \eta \nabla w
        + s^{2} \lambda^{2} \xi^{2} |\nabla \eta|^{2} w
        - s \lambda^{2} \xi |\nabla \eta|^{2} w
        - s \lambda \xi |\Delta \eta| w
    \big)^{2}dxdt
    \\ \notag
    & \leq
    3 \mathbb{E}  \int_{Q}   s^{3} \lambda^{4} \xi^{3} |\nabla \eta|^{4} \big(
        \Delta w
        + s^{2} \lambda^{2} \xi^{2} |\nabla \eta|^{2} w
    \big)^{2}dxdt
    + 12 \mathbb{E}  \int_{Q} s^{5} \lambda^{6} \xi^{5} |\nabla \eta|^{4} |\nabla w\cdot \nabla \eta|^{2}
    + \mathcal{A}_{1}
    .
\end{align}
Combining \cref{eq.s2-1-2,eq.s2-2-2,eq.s2-3-2,eq.s2-4-2,eq.s2-5-2,eq.s2-6-2}, we know there exists $ \lambda_{2} \geq \lambda_{1} $ such that  for all $ \lambda \geq \lambda_{2} $, it holds that
\begin{align}\label{eq.s2-8-2}
    \mathbb{E} \int_{Q} M_{1} dxdt
    + \mathcal{A}
    \geq
    \mathbb{E}  \int_{Q}  \lambda |\nabla \Delta w|^{2} dx dt
    +\mathbb{E}  \int_{Q}  s^{3} \lambda^{4} \xi ^{3} \theta^{2} |\nabla \eta |^{4} | \Delta v|^{2}  dxdt.
\end{align}
It follows from \cref{prop.eta}, \cref{eq.Ic-2,eq.I2c-2,eq.xiEta,eq.AB-2} that
\begin{align}\label{eq.s2-9-2}
    \mathbb{E} \int_{Q} M_{2} dxdt
    +\mathcal{A}
    \geq
    \mathbb{E} \int_{Q} \big(
        16 s \lambda \xi |\nabla \hat{w}|^{2}
        + s^{3} \lambda^{\frac{7}{2}} \xi^{3} |\nabla \eta|^{4} \hat{w}^{2}
\big) dxdt.
\end{align}
Thanks to \cref{eq.s1-2,eq.s2-8-2,eq.s2-9-2}, for $ \lambda \geq \lambda_{2} $, we have that
\begin{align}
    \notag
    & 2 \mathbb{E} \int_{Q} \theta I_{2} (  d\hat{v} + \Delta^{2} v dt ) dx
    + C \mathbb{E} \int_{\Sigma} s^{3} \lambda^{3} \xi^{3} |\Delta w|^{2}  d\Gamma dt
    \\ \notag
    & \quad + C \mathbb{E} \int_{\Sigma} s \lambda \xi |\nabla \Delta w|^{2} d\Gamma dt
    + \mathcal{A}
    + \mathcal{B}_{2}
    \\[1mm] \notag
    &
    + C \mathbb{E} \int_{Q} \theta^{2} \big(
        s^{6} \lambda^{6} \xi^{6} f_{1}^{2}
        + s^{4} \lambda^{4} \xi^{4} |\nabla f_{1}|^{2}
        + s^{2} \lambda^{2} \xi^{2} |\nabla^{2} f_{1}|^{2}
        +s^{6} \lambda^{6} \xi^{6} g_{1}^{2}
    \\[1mm] \label{eq.s2-2-3}
    & \qquad \qquad \quad ~ +s^{4} \lambda^{4} \xi^{4} |\nabla g_{1}|^{2}
    +s^{2} \lambda^{2} \xi^{2} |\nabla^{2} g_{1}|^{2}
    +|\nabla \Delta g_{1}|^{2}
    +s^{2} \lambda^{2} \xi^{2} |g_{2}|^{2}
    \big) dxdt
    \\[1.5mm] \notag
    & \geq
    2 \mathbb{E} \int_{Q}  I_{2}^{2} dxdt
    + 2 \mathbb{E} \int_{Q} I_{2}I_{3} dx  \\ \notag
    & \quad
    +\mathbb{E} \int_{Q} \big(
        \lambda |\nabla \Delta w|^{2}
        + s^{3} \lambda^{4} \xi ^{3} \theta^{2} | \Delta v|^{2}
        +  s \lambda \xi |\nabla \hat{w}|^{2}
        + s^{3} \lambda^{\frac{7}{2}} \xi^{3}  \hat{w}^{2}
    \big) dxdt.
\end{align}

\medskip

\emph{Step 3.} In this step, we get the estimate of $v$.

\smallskip

From \eqref{eq.carlemanEQII}, we have that 
\begin{equation}\label{eq.s3-0-2}
    2 \mathbb{E}\int_{Q} \theta I_{2} \big( d \hat{v} + \Delta^{2} v dt \big) dx
   \leq \mathbb{E} \int_{Q} I_{2}^{2} dxdt + \mathbb{E} \int_{Q} \theta^{2} f_{2}^{2} dx dt.
\end{equation}
Thanks to \cref{prop.eta}, \cref{eq.I3-2,eq.AB-2}, we find that
\begin{align}
    \label{eq.s3-00-2}
    2 \mathbb{E} \int_{Q} I_{2}I_{3}dx  \geq  -  \mathbb{E} \int_{Q} I_{2}^{2} dx dt  - \mathcal{A}
    - C \mathbb{E} \int_{Q} s^{2} \lambda^{2} \xi^{2} \theta^{2} f_{1}^{2} dx dt
    .
\end{align}

Noting that
\begin{align*}
    |\nabla \Delta w|^{2} \leq C \theta^{2} \big( |\nabla \Delta v|^{2} + s^{2}\lambda^{2} \xi^{2} |\nabla^{2} v|^{2}\big) \quad\mbox{ on } \Sigma,
\end{align*}
by \cref{eq.b21-2,eq.BC-2}, there exists $ \lambda_{3} > 0 $ such that  for all  $ \lambda \geq \lambda_{3} $, there is $ s_{1} =  s_{1}(\lambda) > 0 $, such that for all $ s \geq s_{1} $, we have that
\begin{equation}\label{eq.s3-1-2}
\begin{array}{ll}\ds
 \mathcal{B}_{2}
 + \mathbb{E} \int_{\Sigma} s^{3} \lambda^{3} \xi^{3} |\Delta w|^{2}  d\Gamma dt
 + \mathbb{E} \int_{\Sigma} s \lambda \xi |\nabla \Delta w|^{2} d\Gamma dt
 \\
 \ns\ds   \leq
    C \mathbb{E} \int_{\Sigma} s \lambda \xi \theta^{2} \big(
        |\nabla \Delta v|^{2}
 + s^{2}\lambda^{2} \xi^{2} |\Delta v|^{2}
\big) d\Gamma dt.
\end{array}
\end{equation}

Thanks to \cref{eq.s2-2-3,eq.s3-0-2,eq.s3-00-2,eq.s3-1-2}, for $ \lambda \geq \lambda_{3} $ and $ s \geq s_{1} $,  we get that
\begin{align} \label{eq.1126-1}
    \notag
    & C \mathbb{E} \int_{\Sigma} s \lambda \xi \theta^{2} \big(
        |\nabla \Delta v|^{2}
        + s^{2}\lambda^{2} \xi^{2} |\Delta v|^{2}
    \big) d\Gamma dt
    + \mathcal{A}
    \\ 
    &
    + C \mathbb{E} \int_{Q} \theta^{2} \big(
        s^{6} \lambda^{6} \xi^{6} f_{1}^{2}
        + s^{4} \lambda^{4} \xi^{4} |\nabla f_{1}|^{2}
        + s^{2} \lambda^{2} \xi^{2} |\nabla^{2} f_{1}|^{2}
        + f_{2}^{2}
        +s^{6} \lambda^{6} \xi^{6} g_{1}^{2}
    \\ \notag
    & \qquad \qquad \quad ~ +s^{4} \lambda^{4} \xi^{4} |\nabla g_{1}|^{2}
    +s^{2} \lambda^{2} \xi^{2} |\nabla^{2} g_{1}|^{2}
    +|\nabla \Delta g_{1}|^{2}
    +s^{2} \lambda^{2} \xi^{2} |g_{2}|^{2}
    \big) dxdt
    \\[1mm] \notag
    & \geq
    \mathbb{E} \int_{Q} \big(
        \lambda |\nabla \Delta w|^{2}
        + s^{3} \lambda^{4} \xi ^{3} \theta^{2} | \Delta v|^{2}
        +  s \lambda \xi |\nabla \hat{w}|^{2}
        + s^{3} \lambda^{\frac{7}{2}} \xi^{3}  \hat{w}^{2}
    \big) dxdt.
\end{align}

By \cref{lemma.ellipticE-2}, for $ \lambda \geq \max \{ \lambda_{3}, \widehat{\lambda} \} $ and $ s \geq \max \{ s_{1}, \widehat{s} \} $, we obtain that
\begin{align*} 
    \mathbb{E} \int_{Q}  (
        s^{6} \lambda^{8}  \xi^{6} \theta^{2}| v|^{2}
    +   s^{4} \lambda^{6}  \xi^{4} \theta^{2}|\nabla v|^{2}  )dxdt
    \leq     
    C \mathbb{E} \int_{Q}  s^{3} \lambda^{4} \xi ^{3} \theta^{2} | \Delta v|^{2} dx dt,
\end{align*}
which, together with \cref{eq.1126-1},  implies that
\begin{align} \label{eq.1126-3}
    \notag
    & C \mathbb{E} \int_{\Sigma} s \lambda \xi \theta^{2} \big(
        |\nabla \Delta v|^{2}
        + s^{2}\lambda^{2} \xi^{2} |\Delta v|^{2}
    \big) d\Gamma dt
    + \mathcal{A}
    \\ 
    &
    + C \mathbb{E} \int_{Q} \theta^{2} \big(
        s^{6} \lambda^{6} \xi^{6} f_{1}^{2}
        + s^{4} \lambda^{4} \xi^{4} |\nabla f_{1}|^{2}
        + s^{2} \lambda^{2} \xi^{2} |\nabla^{2} f_{1}|^{2}
        + f_{2}^{2}
        +s^{6} \lambda^{6} \xi^{6} g_{1}^{2}
    \\ \notag
    & \qquad \qquad \quad ~ +s^{4} \lambda^{4} \xi^{4} |\nabla g_{1}|^{2}
    +s^{2} \lambda^{2} \xi^{2} |\nabla^{2} g_{1}|^{2}
    +|\nabla \Delta g_{1}|^{2}
    +s^{2} \lambda^{2} \xi^{2} |g_{2}|^{2}
    \big) dxdt
    \\[1mm] \notag
    & \geq
    \mathbb{E} \int_{Q} \big(
        \lambda |\nabla \Delta w|^{2}
        + s^{3} \lambda^{4} \xi ^{3} \theta^{2} | \Delta v|^{2}
        +   s^{4} \lambda^{6}  \xi^{4} \theta^{2}|\nabla v|^{2}
        +  s^{6} \lambda^{8}  \xi^{6} \theta^{2}| v|^{2}
        \\[-1mm] \notag
    & \qquad \qquad 
        +  s \lambda \xi |\nabla \hat{w}|^{2}
        + s^{3} \lambda^{\frac{7}{2}} \xi^{3}  \hat{w}^{2}
    \big) dxdt.
\end{align}

Let $ \tilde{v} = s \lambda^{2} \xi e^{s \xi} v$. Then, we have that
\begin{align}\notag
    &\mathbb{E} \int_{Q}  s^{2} \lambda^{4}  \xi^{2} \theta^{2}| \nabla^{2} v|^{2}  dxdt \\ \notag
    &  
    = \mathbb{E} \int_{Q}  s^{2} \lambda^{4}  \xi^{2} \theta^{2}| \nabla^{2} ( s^{-1} \lambda^{-2} \xi^{-1} \theta^{-1} \tilde{v} )|^{2}  dxdt 
    \\ \notag
    &
    \leq C \mathbb{E} \int_{Q}  s^{2} \lambda^{4}  \xi^{2} ( 
        s^{-2} \lambda^{-4} \xi^{-2} |\nabla^{2} \tilde{v} |^{2}
        + \lambda^{-2}  |\nabla \tilde{v} |^{2}
        + s^{2}  \xi^{2} |\tilde{v} |^{2}
     )  dxdt 
    \\ \label{eq.s3-4-2}
    &
    \leq C |\tilde{v}|^{2}_{L^{2}_{\mathbb{F}}( 0,T; H^{2}(G) )} 
    + C \mathbb{E} \int_{Q}  s^{2} \lambda^{2}  \xi^{2}  |\nabla \tilde{v} |^{2} dxdt 
    + C \mathbb{E} \int_{Q}  s^{4} \lambda^{4}  \xi^{4}  | \tilde{v} |^{2} dxdt 
    \\ \notag
    &
    \leq C |\tilde{v}|^{2}_{L^{2}_{\mathbb{F}}( 0,T; H^{2}(G) )} 
    + C \mathbb{E} \int_{Q}  s^{2} \lambda^{2}  \xi^{2}  ( 
        s^{4} \lambda^{6}  \xi^{4} \theta^{2} v^{2}
        + s^{2} \lambda^{4}  \xi^{2} \theta^{2} |\nabla v|^{2}
        ) dxdt \\ \notag
        & \quad
    + C \mathbb{E} \int_{Q}  s^{6} \lambda^{8}  \xi^{6} \theta^{2} | v |^{2} dxdt 
    \\ \notag
    &
    \leq 
    C |\tilde{v}|^{2}_{L^{2}_{\mathbb{F}}( 0,T; H^{2}(G) )} 
    + C \mathbb{E} \int_{Q}  (
        s^{6} \lambda^{8}  \xi^{6} \theta^{2}| v|^{2}
    +   s^{4} \lambda^{6}  \xi^{4} \theta^{2}|\nabla v|^{2}  )dxdt.
\end{align}
It follows from $ \tilde{v} =0 $ on $ \Sigma $ that
\begin{equation}\label{eq.s3-3-2}
\begin{array}{ll}\ds
        |\tilde{v}|^{2}_{L^{2}_{\mathbb{F}}( 0,T; H^{2}(G) )} \\[1mm]
        \ns\ds \leq C |\Delta \tilde{v}|^{2}_{L^{2}_{\mathbb{F}}( 0,T; L^{2}(G) )}\\
        \ns\ds \leq C \mathbb{E} \int_{Q}  \big(s^{6} \lambda^{8}  \xi^{6} \theta^{2}| v|^{2}+s^{4} \lambda^{6}  \xi^{4} \theta^{2}|\nabla v|^{2} + s^{2} \lambda^{4}  \xi^{2} \theta^{2}| \Delta v|^{2} \big)dxdt.
\end{array}
\end{equation}
 
Combining \cref{eq.s3-4-2,eq.s3-3-2}, we obtain that 
\begin{align}\label{eq.1126-4}
    \mathbb{E} \int_{Q}  s^{2} \lambda^{4}  \xi^{2} \theta^{2}| \nabla^{2} v|^{2}  dxdt
    \leq C \mathbb{E} \int_{Q}  \big(s^{6} \lambda^{8}  \xi^{6} \theta^{2}| v|^{2}+s^{4} \lambda^{6}  \xi^{4} \theta^{2}|\nabla v|^{2} + s^{2} \lambda^{4}  \xi^{2} \theta^{2}| \Delta v|^{2} \big)dxdt.
\end{align}
From \cref{eq.1126-3,eq.1126-4}, there exists $ \lambda_{4} \geq \max \{ \lambda_{3}, \widehat{\lambda} \} $ such that  for all  $ \lambda \geq \lambda_{4} $, there is an $ s_{2} =  s_{2}(\lambda) > \max \{ s_{1}, \widehat{s} \} $, such that for all $ s \geq s_{2} $, we have that
\begin{align}\label{eq.1126-5}
    \notag
    & C \mathbb{E} \int_{\Sigma} s \lambda \xi \theta^{2} \big(
        |\nabla \Delta v|^{2}
        + s^{2}\lambda^{2} \xi^{2} |\Delta v|^{2}
    \big) d\Gamma dt
    + \mathcal{A}
    \\ \notag
    &
    + C \mathbb{E} \int_{Q} \theta^{2} \big(
        s^{6} \lambda^{6} \xi^{6} f_{1}^{2}
        + s^{4} \lambda^{4} \xi^{4} |\nabla f_{1}|^{2}
        + s^{2} \lambda^{2} \xi^{2} |\nabla^{2} f_{1}|^{2}
        + f_{2}^{2}
        +s^{6} \lambda^{6} \xi^{6} g_{1}^{2}
    \\ 
    & \qquad \qquad \quad ~ +s^{4} \lambda^{4} \xi^{4} |\nabla g_{1}|^{2}
    +s^{2} \lambda^{2} \xi^{2} |\nabla^{2} g_{1}|^{2}
    +|\nabla \Delta g_{1}|^{2}
    +s^{2} \lambda^{2} \xi^{2} |g_{2}|^{2}
    \big) dxdt
    \\[1mm] \notag
    & \geq
    \mathbb{E} \int_{Q} \big(
        \lambda |\nabla \Delta w|^{2}
        + s^{3} \lambda^{4} \xi ^{3} \theta^{2} | \Delta v|^{2}
        + s^{2} \lambda^{4}  \xi^{2} \theta^{2}| \nabla^{2} v|^{2} 
        +   s^{4} \lambda^{6}  \xi^{4} \theta^{2}|\nabla v|^{2}
        +  s^{6} \lambda^{8}  \xi^{6} \theta^{2}| v|^{2}
        \\ \notag
    & \qquad \qquad 
        +  s \lambda \xi |\nabla \hat{w}|^{2}
        + s^{3} \lambda^{\frac{7}{2}} \xi^{3}  \hat{w}^{2}
    \big) dxdt.
\end{align}

Recalling $v= \theta^{-1} w$ and $\hat{v}=\theta^{-1}( \hat{w} - \ell_{t} w)$, we get that
\begin{align}\notag
    & \mathbb{E} \int_{Q} \theta^{2} \big(
        \lambda |\nabla \Delta v|^{2}
        +  s \lambda \xi  |\nabla  \hat{v}|^{2}
        +  s^{3} \lambda^{\frac{7}{2}} \xi^{3} |\hat{v}|^{2}
    \big) dxdt
    \\ \label{eq.s3-5-2}
    & = \mathbb{E} \int_{Q} \theta^{2} \big(
        \lambda |\nabla \Delta (\theta^{-1} w)|^{2}
        +  s \lambda \xi  |\nabla  [\theta^{-1}( \hat{w} - \ell_{t} w)]|^{2}
        +  s^{3} \lambda^{\frac{7}{2}} \xi^{3} \theta^{-2} | \hat{w} - \ell_{t} w|^{2}
    \big) dxdt
    \\ \notag
    & \leq
    C \mathbb{E} \int_{Q} \big[
        \lambda \big(
            |\nabla \Delta w|^{2}
            + s^{2} \lambda^{2} \xi^{2} |\nabla^{2} w|^{2}
            + s^{4} \lambda^{4} \xi^{4} |\nabla w|^{2}
            + s^{6} \lambda^{6} \xi^{6} | w|^{2}
        \big)
        +  s^{3} \lambda^{\frac{7}{2}} \xi^{3} \big(
            |\hat{w}|^{2}
            + s^{2} \lambda^{2} \xi^{2} | w|^{2}
        \big)
        \\[1mm] \notag
        & \qquad \qquad ~
        + s \lambda \xi \big(
            | \nabla \hat{w}|^{2}
            +  s^{2} \lambda^{2} \xi^{2} |\hat{w}|^{2}
            +  s^{2} \lambda^{2} \xi^{2} |\nabla w|^{2}
            +  s^{4} \lambda^{4} \xi^{4} | w|^{2}
        \big)
    \big] dxdt
    \\[2mm] \notag
    & \leq
    C \mathbb{E} \int_{Q} \big(
        \lambda |\nabla \Delta w|^{2}
        +  s \lambda \xi  |\nabla  \hat{w}|^{2}
        +  s^{3} \lambda^{\frac{7}{2}} \xi^{3} |\hat{w}|^{2}
    \big) dxdt
    + \mathcal{A}.
\end{align}
Thanks to \cref{eq.AB-2,prop.eta}, there exists $ \lambda_{5} > 0 $ such that  for all  $ \lambda \geq \lambda_{5} $, there is an $ s_{3} = s_{3}(\lambda) > 0 $, such that for all $ s \geq s_{3} $, we have
\begin{align}\notag
    \mathcal{A}
    &\leq \frac{1}{C} \mathbb{E} \int_{Q} \theta^{2} \bigl(
        s \lambda \xi |\nabla \hat{v}|^{2}
      + s^{3} \lambda^{\frac{7}{2}} \xi^{3} | \hat{v}|^{2}
      + s^{2} \lambda^{4} \xi^{2} |\nabla^{2} v|^{2}
      + s^{3} \lambda^{4} \xi^{3} |\Delta  v|^{2}
      \\\label{eq.s3-6-2}
      & \qquad \qquad \quad\;
      + s^{4} \lambda^{6} \xi^{4} |\nabla  v|^{2}
      + s^{6} \lambda^{8} \xi^{6} |  v|^{2}
  \bigr) dxdt.
\end{align}

Let us choose $ \lambda_{0} \geq \max\{\lambda_{4}, \lambda_{5} \} $. Combining \cref{eq.1126-5,eq.s3-5-2,eq.s3-6-2}, for all $ \lambda \geq \lambda_{0} $, one can find $ s_{0} = s_{0}(\lambda) \geq \max\{ s_{2}, s_{3}  \} $ so that for any $ s \geq s_{0} $, inequality \cref{eq.08012022} holds.
\end{proof}

\section{Proof of the observability estimate}\label{sc.4}

\begin{proof}[Proof of \cref{thm.ObeF}]
  Let $ \chi \in C_{0}^{\infty}([0,T]) $ satisfy 
    \begin{equation*}
        \chi = 1 \text{ in } \bigg(\frac{T}{2}-\varepsilon_{1}T, \frac{T}{2}+\varepsilon_{1}T \bigg).
    \end{equation*}
    Put $ v = \chi z $ and $ \hat{v} = \chi \hat{z} + \chi_{t}z $ for $ (z, \hat{z})  $ satisfying \cref{eq.reference}, then $ (v, \hat{v}) $ fulfills $v(0,\cdot)=v(T,\cdot)=\hat{v}(0,\cdot)=\hat{v}(T,\cdot)=0$    in $G$, and solves
    \begin{equation}\label{eq.hII2F}
        \left\{
        \begin{array}{ll}\ds
           d v = \hat{v} dt + \chi (Z- a_{5} z) d W(t) &  \text { in } Q, \\
          \ns\ds d \hat{v}+\Delta^2 v d t= \tilde{f}_{2} d t + \tilde{g}_{2}  d W(t) &  \text { in } Q, \\
          \ns\ds v= \frac{\partial v}{ \partial \nu}=0 & \text { on } \Sigma,
        \end{array}
        \right.
    \end{equation}
    where $$ \tilde{f}_{2} = \chi [(a_{1} - \operatorname{div} a_{2} -a_{4}a_{5}) z - a_{2} \nabla z- a_{3} \hat{Z} + a_{4} Z ] + 2 \chi_{t} \hat{z} + \chi_{tt} z $$ and $$ \tilde{g}_{2} = \chi \hat{Z} +   \chi_{t} (Z - a_{5} z ).  $$

By \cref{thm.ceII}, for $ \lambda \geq \lambda_{0} $ and $ s \geq s_{0} $, we have
    \begin{align*}
        &\mathbb{E} \int_{Q} \theta^{2} \chi^{2} \bigl(
            \lambda |\nabla \Delta z|^{2}
            + s^{2} \lambda^{4} \xi^{2} |\nabla^{2} z|^{2}
            + s^{3} \lambda^{4} \xi^{3} |\Delta  z|^{2}
            + s^{4} \lambda^{6} \xi^{4} |\nabla  z|^{2}
            + s^{6} \lambda^{8} \xi^{6} |  z|^{2}
        \bigr) dxdt
        \\ \notag
        & \leq \mathbb{E} \int_{Q} \theta^{2} \bigl(
              s \lambda \xi |\nabla \hat{v}|^{2}
            + s^{3} \lambda^{\frac{7}{2}} \xi^{3} | \hat{v}|^{2}
            +  \lambda |\nabla \Delta v|^{2}
            + s^{2} \lambda^{4} \xi^{2} |\nabla^{2} v|^{2}
            + s^{3} \lambda^{4} \xi^{3} |\Delta  v|^{2}
            \\[1mm] \notag
            & \qquad \qquad ~ + s^{4} \lambda^{6} \xi^{4} |\nabla  v|^{2}
            + s^{6} \lambda^{8} \xi^{6} |v|^{2}
        \bigr) dxdt
        \\[2mm] \notag
        &
        \leq
        C \mathbb{E} \int_{Q} \theta^{2} \chi^{2} \big(
            s^{6} \lambda^{6} \xi^{6} z^{2}
            +s^{4} \lambda^{4} \xi^{4} |\nabla z|^{2}
            +s^{2} \lambda^{2} \xi^{2} |\nabla^{2} z|^{2}
            +|\nabla \Delta z|^{2}
            + |z|^{2}
            + |\nabla z|^{2}
            \big) dxdt
            \\ \notag
            & \quad
        + C \mathbb{E} \int_{J_{1}} \theta^{2} \bigr(
        |\hat{z}|^{2} + z^{2}  + s^{2} \lambda ^{2} \xi^{2} z^{2}
        \bigl)dx dt
        + C (s, \lambda) |(Z, \hat{Z})|^{2}_{L^{2}_{\mathbb{F}}(0, T;H^{3}(G) )\times L^{2}_{\mathbb{F}}(0, T; H^{1}(G))}
            \\
            & \quad
        + C \mathbb{E} \int_{\Sigma} \theta^{2} \bigr(
            s \lambda \xi |\nabla \Delta z|^{2}
            + s^{3} \lambda^{3} \xi^{3} |\Delta z|^{2}
        \bigl)d\Gamma dt
        .
    \end{align*}
    This, together with  \cref{prop.eta}, implies that there exists $ \tilde{\lambda}_{1} \geq  \lambda_{0} $ such that  for all  $ \lambda \geq \tilde{\lambda}_{1} $, there is  $ \tilde{s}_{1} = \tilde{s}_{1}(\lambda) \geq s_{0} $, so that for any $ s \geq \tilde{s}_{1}$, it holds that 
    \begin{align} \notag
        & \mathbb{E} \int_{Q} \theta^{2} \bigl(
              s \lambda \xi |\nabla \hat{v}|^{2}
            + s^{3} \lambda^{\frac{7}{2}} \xi^{3} | \hat{v}|^{2}
            +  \lambda |\nabla \Delta v|^{2}
            + s^{2} \lambda^{4} \xi^{2} |\nabla^{2} v|^{2}
            + s^{3} \lambda^{4} \xi^{3} |\Delta v|^{2}
            \\ \label{eq.thm1.5.1F}
            & \qquad \quad + s^{4} \lambda^{6} \xi^{4} |\nabla v|^{2}
            + s^{6} \lambda^{8} \xi^{6} |v|^{2}
        \bigr) dxdt
        \\[2mm] \notag
        & \leq C  \mathbb{E} \int_{J_{1}} \theta^{2} \bigr(
            |\hat{z}|^{2} + z^{2}  + s^{2} \lambda ^{2} \xi^{2} z^{2}
        \bigl)dx dt
        + C \mathbb{E} \int_{\Sigma} \theta^{2} \bigr(
            s \lambda \xi |\nabla \Delta z|^{2}
            + s^{3} \lambda^{3} \xi^{3} |\Delta z|^{2}
        \bigl)d\Gamma dt
        \\ \notag
        & \quad
        + C (s, \lambda) |(Z, \hat{Z})|^{2}_{L^{2}_{\mathbb{F}}(0, T;H^{3}(G) )\times L^{2}_{\mathbb{F}}(0, T; H^{1}(G))}
        .
    \end{align}
    Thanks to \cref{prop.eta}, we obtain
    \begin{align} \notag
        & e^{2se^{\lambda \delta} - C \lambda - 6 \ln s}\mathbb{E} \int_{J_{2}}  \bigl(
               |\nabla \Delta z|^{2} +  |\nabla ^{2} z|^{2} + |\nabla z|^{2} + z^{2}
               + |\nabla \hat{z} |^{2} + |\hat{z}|^{2}
            \bigr) dxdt
        \\ \notag
        & \leq
        e^{2se^{\lambda \delta}}  \mathbb{E} \int_{J_{2}}  \bigl(
            s \lambda \xi |\nabla \hat{v}|^{2}
          + s^{3} \lambda^{\frac{7}{2}} \xi^{3} | \hat{v}|^{2}
          +  \lambda |\nabla \Delta v|^{2}
          + s^{2} \lambda^{4} \xi^{2} |\nabla^{2} v|^{2}
          + s^{3} \lambda^{4} \xi^{3} |\Delta  v|^{2}
          \\[1mm] \label{eq.obEP12F}
          & \qquad \qquad \qquad + s^{4} \lambda^{6} \xi^{4} |\nabla v|^{2}
          + s^{6} \lambda^{8} \xi^{6} |v|^{2}
      \bigr) dxdt
        \\[1.5mm]  \notag
        & \leq
        \mathbb{E} \int_{Q} \theta^{2} \bigl(
              s \lambda \xi |\nabla \hat{v}|^{2}
            + s^{3} \lambda^{\frac{7}{2}} \xi^{3} | \hat{v}|^{2}
            +  \lambda |\nabla \Delta v|^{2}
            + s^{2} \lambda^{4} \xi^{2} |\nabla^{2} v|^{2}
            + s^{3} \lambda^{4} \xi^{3} |\Delta v|^{2} 
            \\[1mm] \notag
            & \qquad \qquad ~ + s^{4} \lambda^{6} \xi^{4} |\nabla v|^{2}
            + s^{6} \lambda^{8} \xi^{6} |v|^{2}
        \bigr) dxdt.
    \end{align}
From \cref{prop.energyEst,prop.eta}, we see that
\begin{align}  \notag
   & |(z^{T},\hat{z}^{T})|_{L^{2}_{\mathcal{F}_{T}}(\Omega; H^{3}(G) \cap H_{0}^{2}(G) )\times L^{2}_{\mathcal{F}_{T}}(\Omega; H_{0}^{1}(G))}\\ \label{eq.obEP22F}
    & \leq
    C
        \mathbb{E} \int_{J_{2}}  \bigl(
              |\nabla \Delta z|^{2} +  |\nabla ^{2} z|^{2} + |\nabla z|^{2} + z^{2}
              + |\nabla \hat{z}|^{2} + |\hat{z}|^{2}
            \bigr) dxdt
    \\ \notag
    &\quad   + C (s, \lambda) |(Z, \hat{Z})|^{2}_{L^{2}_{\mathbb{F}}(0, T;H^{3}(G) \cap H_{0}^{2}(G))\times L^{2}_{\mathbb{F}}(0, T; H_{0}^{1}(G))}
    ,
\end{align}
and that
\begin{eqnarray}\label{eq.obEP32F}
&&\mathbb{E} \int_{J_{1}} \theta^{2} \big(|\hat{z}|^{2} + z^{2}  + s^{2} \lambda ^{2} \xi^{2} z^{2}\big) dxdt \nonumber\\
&& \ds\leq e^{2s e^{-\lambda \delta}  + C \lambda + 2 \ln s} \mathbb{E} \int_{J_{1}} \big(|\hat{z}|^{2} + z^{2} \big) dxdt
    \\ 
&&\leq C e^{2s e^{-\lambda \delta}  + C \lambda + 2 \ln s }
    |(z^{T},\hat{z}^{T})|_{L^{2}_{\mathcal{F}_{T}}(\Omega; H^{3}(G) \cap H_{0}^{2}(G) )\times L^{2}_{\mathcal{F}_{T}}(\Omega; H_{0}^{1}(G))}\nonumber\\[0.5mm]
&&\quad +  C (s, \lambda) |(Z, \hat{Z})|^{2}_{L^{2}_{\mathbb{F}}(0, T;H^{3}(G) \cap H_{0}^{2}(G) ) \times L^{2}_{\mathbb{F}}(0, T;H_{0}^{1}(G) )}.\nonumber
\end{eqnarray}

Combining \eqref{eq.thm1.5.1F}--\eqref{eq.obEP32F} choosing $ \lambda \geq \tilde{\lambda}_{1} $ and $ s \geq \tilde{s}_{1} $ such that
\begin{align*}
    C \exp\big(2 s e^{-\lambda \delta} -2 s e^{\lambda \delta} + C \lambda + 8 \ln s\big) \leq \frac{1}{2},
\end{align*}
we get the desired observability estimate.
\end{proof}

\section{Proof of Theorem 1.4}
\label{sc.5}

The proof of \cref{thm.negativeControl} is similar to that for \cite[Theorem 2.3]{Lue2019}. We provide it here for the convenience of the readers. To begin with, we recall the following result.
\begin{lemma}\label{lemma.peng}\cite[Lemma 2.1]{Peng1994}
    There exists a random variable $ \zeta \in L^{2}_{\mathcal{F}_{T}} (\Omega) $ such that it is impossible to find $ \varsigma_{1}, \varsigma_{2} \in L^{2}_{\mathbb{F}}(0,T) \times C_{\mathbb{F}}([0,T]; L^{2}(\Omega)) $ and $ \alpha \in \mathbb{R} $ satisfying
    \begin{align*}
        \zeta = \alpha + \int_{0}^{T} \varsigma_{1}(t) dt + \int_{0}^{T} \varsigma_{2}(t) dW(t).
    \end{align*}
\end{lemma}

\begin{proof}[Proof of \cref{thm.negativeControl}]

We employ a contradiction argument and divide the proof into three cases.

\medskip

\emph{Case 1:  $ a_{3} \in C_{\mathbb{F}} ([0,T]; L^{\infty}(G)) $, $ G \backslash \overline{G_{0}} \neq \emptyset $ and $ \operatorname{supp} f \subset G_{0} $}.

\smallskip

Let $ \rho \in C_{0}^{\infty}(G\backslash G_{0}) $ satisfying $ |\rho|_{L^{2}(G)} = 1 $.
Suppose that \cref{eq.eqIIICon} was exactly controllable. By \cref{def.control}, for $ (y_{0}, \hat{y}_{0}) = (0,0) $, there exist  controls $ (f, g, h_{1}, h_{2})  $ with $ \operatorname{supp} f \subset G_{0} $ a.e. $ (t,\omega) \in (0,T) \times \Omega $ such that the solution to \cref{eq.eqIIICon} fulfills $ (y(T), \hat{y}(T)) = (\rho \zeta, 0) $, where $ \zeta $ is given in \cref{lemma.peng}.
Hence,
\begin{align}
    \label{eq.noLessControl1}
    \rho \zeta = \int_{0}^{T} \hat{y} dt + \int_{0}^{T} (a_{3}y + f) dW(t).
\end{align}
Multiplying \cref{eq.noLessControl1} by $ \rho $ and integrating it in $ G $, we arrive that
\begin{align*}
    \zeta = \int_{0}^{T} \langle \hat{y}, \rho \rangle _{(H^{3}(G)\cap H_{0}^{2}(G))^{*}, H^{3}(G)\cap H_{0}^{2}(G) }  dt + \int_{0}^{T} \langle a_{3}y , \rho \rangle_{H^{-1}(G), H_{0}^{1}(G)} dW(t),
\end{align*}
which contradicts \cref{lemma.peng}.

\medskip

\emph{Case 2:  $ a_{4} \in C_{\mathbb{F}} ([0,T]; L^{\infty}(G)) $, $ G \backslash \overline{G_{0}} \neq \emptyset $ and $ \operatorname{supp} g \subset G_{0} $}.

\smallskip

Choose $ \rho $ as in Case 1.
Assume that \cref{eq.eqIIICon} was exactly controllable. Then, for $ (y_{0}, \hat{y}_{0}) = (0,0) $, there exist controls $ (f, g, h_{1}, h_{2})  $ with $ \operatorname{supp} g \subset G_{0} $ a.e. $ (t,\omega) \in (0,T) \times \Omega $ such that the solution to \cref{eq.eqIIICon} fulfills $ (y(T), \hat{y}(T)) = (0, \zeta) $.

Clearly, $ (\phi, \hat{\phi}) \deq (\rho y, \rho \hat{y}) $ satisfies
\begin{align*}
    \left\{
    \begin{alignedat}{2}
        &d \phi = \hat{\phi} d t +(a_{3} \phi + \rho f) d W(t) && \quad \text { in } Q, \\
        &d \hat{\phi} +\Delta^2 \phi d t = \tilde{f}_{2}  d t + a_{4} \phi d W(t) && \quad \text { in } Q, \\
        &\phi = \frac{\partial \phi}{ \partial \nu}= 0 && \quad \text { on } \Sigma, \\
        & (\phi(0), \hat{\phi}(0)) = (0,0) && \quad \text{ in } G,
    \end{alignedat}
    \right.
\end{align*}
where $ \tilde{f}_{2} = [\Delta^{2} , \rho] y + a_{1} \phi + \rho a_{2} \cdot \nabla \phi  $.
Furthermore, we have $ (\phi(T), \hat{\phi}(T)) = (0, \rho \zeta) $.
Hence, we have
\begin{align*}
    \zeta = - \int_{0}^{T} ( \langle \Delta^{2} \phi, \rho \rangle_{H^{-5}(G), H_{0}^{5}(G)} + \langle \tilde{f}_{2}, \rho \rangle_{H^{-4} (G), H_{0}^{4}(G)} ) dt
    + \int_{0}^{T} \langle a_{3} \phi, \rho \rangle_{H^{-1} (G), H_{0}^{1}(G)} d W(t),
\end{align*}
which contradicts \cref{lemma.peng}.

\medskip

\emph{Case 3: $ h_{1} = h_{2} = 0 $.}

\smallskip

Assume that \cref{eq.eqIIICon} was exactly controllable.
Then, from the equivalence between the exact controllability of \cref{eq.eqIIICon} and the observability estimate of \cref{eq.reference}, we get that for any $ (z^{T},\hat{z}^{T}) \in L^{2}_{\mathcal{F}_{T}}(\Omega; H^{3}(G) \cap H_{0}^{2}(G)) \times  L^{2}_{\mathcal{F}_{T}}(\Omega; H_{0}^{1} (G) )  $, the solution $ (z,Z, \hat {z}, \hat{Z}) $ to \cref{eq.reference}
( with $ \tau = T $ and $ (z(T), \hat{z}(T)) = (z^{T},\hat{z}^{T}) $) satisfies
\begin{align} \label{eq.non2}
    | (z^{T},\hat{z}^{T}) |_{L^{2}_{\mathcal{F}_{T}}(\Omega; H^{3}(G) \cap H_{0}^{2}(G)) \times  L^{2}_{\mathcal{F}_{T}}(\Omega; H_{0}^{1} (G) ) }
    \leq
    C (
        |Z|_{ L^{2}_{\mathbb{F}}(0,T; H^{3}(G) \cap H^{2}_{0}(G) )}
        + |\hat{Z}|_{ L^{2}_{\mathbb{F}}(0, T; H_{0}^{1} (G) )}
    )
    .
\end{align}
For any nonzero $ (\Phi_{0}, \Phi_{1}) \in (H^{3}(G) \cap H_{0}^{2}(G)) \times H_{0}^{1} (G) $, let $ (\Phi, \hat{\Phi}) $ solve the equation
\vspace{-1mm}
\begin{align*}
    \left\{
        \begin{alignedat}{2}
            &d \Phi = \hat{\Phi} d t  - a_{5} \Phi d W(t) && \quad \text { in } Q, \\
            &d \hat{\Phi} +\Delta^2 \Phi d t = [ (a_{1} - \operatorname{div} a_{2} - a_{4} a_{5}  ) \Phi - a_{2} \cdot \nabla \Phi  ]dt && \quad \text { in } Q, \\
            &\Phi= \frac{\partial \Phi}{ \partial \nu}= 0 && \quad \text { on } \Sigma, \\
            & (\Phi(0), \hat{\Phi}(0)) = (\Phi_{0}, \Phi_{1}) && \quad \text{ in } G
            .
        \end{alignedat}
        \right.
\end{align*}
Clearly, $ (\Phi,0,\hat{\Phi},0) $ solves \cref{eq.reference} with the final datum $  (z^{T}, \hat{z}^{T}) = (\Phi(T), \hat{\Phi}(T))  $, a contradiction to the inequality \cref{eq.non2}.
\end{proof}

\appendix

\section{Proof of the weighted identity}\label{sc.a}

\begin{proof}[Proof of \cref{thm.fundamentalIndentity-3}]
    It is clear that
    \begin{align*}
        d w  = d (\theta v) = \theta d v + \ell_{t} \theta v dt  = \hat{w}  dt + \theta f_{1} dt + \theta g_{1} d W(t) ,
    \end{align*}
    and
    \begin{align} \label{eq.dht-2}
        \notag
        \theta d \hat{v}
        & = \theta d[ \theta^{-1} (\hat{w} - \ell_{t} w) ]
        =
        d \hat{w} - \ell_{t} dw -\ell_{tt} w dt - \ell_{t} \hat{w} dt + \ell_{t}^{2} dt
        \\
        & = d \hat{w} - 2 \ell_{t} \hat{w} dt + (\ell_{t}^{2} - \ell_{tt}) w dt - \ell_{t} \theta f_{1} dt - \ell_{t} \theta g_{t} d W(t)
        .
    \end{align}
    We also have
    \begin{align}
        \notag
    \theta \Delta^{2} v
        & =
        \Delta^{2} w
        -4 s\lambda\xi \nabla \eta \cdot \nabla \Delta w
        -4 s\lambda^{2}\xi  (\nabla^{2}w\nabla\eta\nabla\eta )
        -4 s\lambda\xi \sum\limits_{i,j=1}^{n} \eta_{x_{i}x_{j}} w_{x_{i}x_{j}}
        \\ \notag
        & \quad
        + 2 s^{2}\lambda^{2}\xi^{2} |\nabla\eta|^{2}\Delta w
        -2 s\lambda^{2}\xi |\nabla \eta|^{2} \Delta w  - 2s\lambda\xi  \Delta \eta \Delta w
        + 4 s^{2}\lambda^{2}\xi^{2} ( \nabla^{2}w\nabla\eta\nabla\eta  )
        \\[3mm] \label{eq.space-2}
        & \quad
        -4 \nabla \Delta \ell \cdot \nabla w
        +12 s^{2}\lambda^{3}\xi^{2} |\nabla\eta|^{2} \nabla\eta \nabla w
        +8 s^{2}\lambda^{2}\xi^{2}  ( \nabla^{2}\eta \nabla \eta \nabla w  )
        \\[3mm] \notag
        & \quad
        -4 s^{3} \lambda^{3} \xi^{3} |\nabla\eta|^{2} \nabla\eta \nabla w
        +4 s^{2} \lambda^{2} \xi^{2} \Delta \eta \nabla \eta \nabla w
        + 4  (\nabla \ell \cdot \nabla \Delta \ell ) w
        +2 |\nabla^{2} \ell|^{2} w -\Delta^{2} \ell w
        \\[3mm] \notag
        & \quad
        - 6 s^{3} \lambda ^{4}\xi^{3} |\nabla\eta|^{4} w
        - 4 s ^{3}\lambda ^{3}\xi^{3}  (\nabla^{2}\eta \nabla \eta \nabla \eta ) w
        + s ^{4}\lambda ^{4}\xi^{4} |\nabla \eta|^{4} w
        - 2 s^{3} \lambda ^{3}\xi^{3} |\nabla\eta|^{2} \Delta \eta w 
        + |\Delta \ell|^{2} w.
    \end{align}
    From \cref{eq.I-2,eq.Ic-2,eq.I2-2,eq.I2c-2,eq.dht-2,eq.I3-2,eq.space-2}, we have
    \begin{equation*}
        2\theta I_{2} (d \hat{v} +\Delta^{2} v dt)=2 I_{2} ( I_{1} + I_{2} dt+ I_{3}  ).
    \end{equation*}

    We will compute $ I_{1} I_{2} $ under the form $\sum\limits_{i=1}^{7} \sum\limits_{j=1}^{5} I_{i j}$, where $I_{i j}$ is the product of the $i$-th term of $I_{1}$ with the $j$-th term of $I_{2}$.
    Note that $ I_{i j} $ are the same as Appendix A in \cite{Wang} for $ i=1,\cdots, 6 $ and $ j=1,\cdots,5 $, except for $ I_{13} $.

    We have
    \begin{align*}
        I_{13}
        &= \sum_{i,j,k,l=1}^{n} \Phi_{3}^{kl} w_{x_{i}x_{i}x_{j}x_{j}} w_{x_{k}x_{l}} dt
        \\
        & =
        \sum_{i,j,k,l=1}^{n} \biggl(
            \Phi_{3}^{kl} w_{x_{i}x_{i}x_{j}} w_{x_{k}x_{l}}
            - \Phi_{3}^{kj} w_{x_{i}x_{i}x_{l}} w_{x_{k}x_{l}}
            + \Phi_{3}^{kl} w_{x_{i}x_{i}x_{l}} w_{x_{k}x_{j}}
            - \Phi_{3_{x_{i}}}^{kl} w_{x_{i}x_{j}} w_{x_{k}x_{l}}
            \\
            &\qquad \qquad ~~
            + \Phi_{3_{x_{l}}}^{kj} w_{x_{i}x_{l}} w_{x_{i}x_{k}}
            + \Phi_{3_{x_{l}}}^{kl} w_{x_{i}x_{j}} w_{x_{i}x_{k}}
            -\frac{1}{2} \Phi_{3_{x_{l}}}^{jl} w_{x_{i}x_{k}}^{2}
            - \Phi_{3_{x_{i}}}^{kl} w_{x_{j}x_{l}} w_{x_{i}x_{k}}
            \biggr)_{x_{j}} dt
        \\
        & \quad
        - \sum_{i,j,k,l=1}^{n} \Phi_{3}^{kl}w_{x_{i}x_{i}x_{l}}w_{x_{k}x_{j}x_{j}} dt
        + \sum_{i,j,k,l=1}^{n} \Phi_{3x_{i}x_{j}}^{kl} w_{x_{i}x_{j}}w_{x_{k}x_{l}} dt
        - \sum_{i,j,k,l=1}^{n} \Phi_{3x_{l}x_{j}} ^{kl} w_{x_{i}x_{j}}w_{x_{i}x_{k}} dt
        \\
        & \quad 
        - \sum_{i,j,k,l=1}^{n} \Phi_{3x_{i}x_{l}} ^{kl} w_{x_{i}x_{j}}w_{x_{k}x_{j}} dt 
        + \sum_{i,j,k,l=1}^{n} \Phi_{3x_{i}x_{j}} ^{kl} w_{x_{i}x_{l}}w_{x_{k}x_{j}} dt
        + \frac{1}{2} \sum_{i,j,k,l=1}^{n} \Phi_{3x_{k}x_{l}} ^{kl} w_{x_{i}x_{j}}^{2} dt
        .
    \end{align*}

    We have
    \begin{align*}
        I_{71}
        &= \sum_{i,j=1}^{n} \Phi_{1}^{i}w_{x_{i}x_{j}x_{j}} d \hat{w}
        \\
        & =
        \sum_{i,j=1}^{n} \biggl(
            - \Phi_{1}^{j} \hat{w} d w_{x_{i}x_{i}}
            + \Phi_{1 x_{i}}^{i} \hat{w}  \hat{w}_{x_{j}}
            - \frac{1}{2} \Phi_{1 x_{i} x_{j}}^{i} \hat{w}^{2}
            + \Phi_{1 }^{i} \hat{w}_{x_{i}}  \hat{w}_{x_{j}}
            - \frac{1}{2} \Phi_{1}^{j} \hat {w} _{x_{i}}^{2}
        \biggr)_{x_{j}}dt
        \\
        & \quad
        + \sum_{i,j=1}^{n} d (\Phi^{i}_{1} w_{x_{i}x_{j}x_{j}} \hat{w})
        + \frac{1}{2} \sum_{i,j=1}^{n} \Phi_{1x_{i}x_{j}x_{j}}^{i} \hat{w}^{2} dt
        - \frac{1}{2} \sum_{i,j=1}^{n} \Phi_{1x_{i}}^{i} \hat{w}_{x_{j}}^{2}  dt
        -  \sum_{i,j=1}^{n} \Phi_{1x_{j}}^{i} \hat{w}_{x_{i}}\hat{w}_{x_{j}}  dt
        \\
        & \quad
        - \sum_{i,j=1}^{n} \Phi_{1t}^{i} w_{x_{i}x_{j}x_{j}} \hat{w} dt
        + \sum_{i,j=1}^{n} (\Phi_{1x_{i}}^{i} \hat{w} + \Phi_{1}^{i} \hat{w}_{x_{i}}) [ (\theta f_{1})_{x_{j}x_{j}} dt +  (\theta g_{1})_{x_{j}x_{j}} d W(t) ]
        \\
        & \quad
        - \sum_{i,j=1}^{n} \Phi_{1}^{i} d w_{x_{i}x_{j}x_{j}} d \hat{w}
        ,
        \\
        I_{72}
        &= \sum_{i=1}^{n} \Phi_{2} w_{x_{i}x_{i}} d \hat{w}
        \\
        &=
        \sum_{j=1}^{n} \biggl(
            -\Phi_{2} \hat{w} \hat{w}_{x_{j}}
            + \frac{1}{2} \Phi_{2x_{j}} \hat{w} ^{2}
        \biggr)_{x_{j}}dt
        + \sum_{i=1}^{n} d (\Phi_{2} w_{x_{i}x_{i}} \hat{w} )
        - \sum_{i=1}^{n} \Phi_{2t} w_{x_{i}x_{i}} \hat{w} dt
        \\
        & \quad 
        - \frac{1}{2} \sum_{i=1}^{n} \Phi_{2x_{i}x_{i}} \hat{w}^{2}  dt
        + \sum_{i=1}^{n} \Phi_{2} \hat{w}_{x_{i}}^{2}  dt
        - \sum_{i=1}^{n} \Phi_{2} d w_{x_{i}x_{i}} d \hat{w}
        \\
        & \quad - \sum_{i=1}^{n}  \Phi_{2} \hat{w}  [ (\theta f_{1})_{x_{i}x_{i}} dt +  (\theta g_{1})_{x_{i}x_{i}} d W(t) ]
        ,
        \\
        I_{73}
        & = \sum_{i,j=1}^{n} \Phi_{3}^{ij} w_{x_{i}x_{j}} d \hat{w}
        \\
        & =
        \sum_{i,j=1}^{n} \biggl(
            - \Phi_{3}^{ij} \hat{w} \hat{w}_{x_{i}}
            + \frac{1}{2} \Phi_{3x_{i}}^{ij} \hat{w}^{2}
        \biggr)_{x_{j}} dt
        + \sum_{i,j=1}^{n} d (\Phi_{3}^{ij} w_{x_{i}x_{j}} \hat{w})
        - \sum_{i,j=1}^{n} \Phi_{3t}^{ij} w_{x_{i}x_{j}} \hat{w}dt
        \\
        & \quad
        - \frac{1}{2} \sum_{i,j=1}^{n} \Phi_{3x_{i}x_{j}}^{ij} \hat{w}^{2}dt
        - \sum_{i,j=1}^{n} \Phi_{3}^{ij} d w_{x_{i}x_{j}} d \hat{w}
        + \sum_{i,j=1}^{n} \Phi_{3}^{ij} \hat{w}_{x_{i}} \hat{w}_{x_{j}}dt
        \\
        & \quad
        - \sum_{i,j=1}^{n} \Phi_{3}^{ij} \hat{w}  [ (\theta f_{1})_{x_{i}x_{j}} dt +  (\theta g_{1})_{x_{i}x_{j}} d W(t) ]
        ,
        \\
        I_{74}
        &= \sum_{i=1}^{n} \Phi_{4}^{i} w_{x_{i}} d \hat{w}
        \\
        &= \sum_{j=1}^{n} \biggl(-
            \frac{1}{2} \Phi_{4}^{j}  \hat{w}^{2}
        \biggr)_{x_{j}} dt
        + \sum_{i=1}^{n} d (\Phi_{4}^{i} w_{x_{i}}  \hat{w})
        - \sum_{i=1}^{n} \Phi_{4t}^{i} w_{x_{i}}  \hat{w} dt
        +\frac{1}{2} \sum_{i=1}^{n} \Phi_{4x_{i}}^{i}  \hat{w}^{2} dt
        \\
        & \quad 
        - \sum_{i=1}^{n} \Phi^{i}_{4} d w_{x_{i}} d\hat{w}
        - \sum_{i=1}^{n}  \Phi_{4}^{i} \hat{w}  [ (\theta f_{1})_{x_{i}} dt +  (\theta g_{1})_{x_{i}} d W(t) ]
        ,
        \\
        I_{75}
        &= \Phi_{5} w d\hat{w} = d(\Phi_{5} w \hat{w})
        - \Phi_{5t} w \hat{w} dt
        - \Phi_{5}  \hat{w}^{2} dt
        - \Phi_{5} d w d\hat{w}
        - \Phi_{5} \hat{w}   (\theta f_{1} dt +  \theta g_{1} d W(t))
        .
    \end{align*}

    By summing all the $ I_{ij} $,  we get \cref{eq.finalEquation-3}.
\end{proof}

\end{document}